\documentclass[a4paper]{article}

\usepackage{amsfonts}
\usepackage{amsmath}
\usepackage{mathrsfs}
\usepackage{amssymb,color}

\usepackage[latin1]{inputenc}
\usepackage{amsthm}
\usepackage{graphicx}
\usepackage{enumerate}
\newtheorem{theorem}{Theorem}[section]

\newtheorem{definition}[theorem]{Definition}
\newtheorem{lemma}[theorem]{Lemma}
\newtheorem{remark}[theorem]{Remark}
\newtheorem{corollary}[theorem]{Corollary}

\newcommand{\R}{{{\mathbb R}}}

\begin{document}
\title{\textbf{ The eigenvalue Characterization for the constant Sign Green's Functions of  $(k,n-k)$ problems.}}
\date{}
\author{Alberto Cabada\footnote{Partially supported by Ministerio de Educaci\'on y Ciencia, Spain and FEDER, projects MTM2010-15314 and MTM2013-43014-P.}   \;and  Lorena Saavedra\footnote{Supported by Plan I2C scholarship, Conselleria de Educaci\'on, Cultura e O.U., Xunta de Galicia, and FPU scholarship, Ministerio de Educaci\'on, Cultura y Deporte, Spain.}\\Departmento de An\'alise Matem\'atica,\\ Facultade de Matem\'aticas,\\Universidade de Santiago de Compostela,\\ Santiago de Compostela, Galicia,
	Spain\\
	alberto.cabada@usc.es, lorena.saavedra@usc.es
} 
\maketitle

\begin{abstract}
	This paper is devoted to the study of the sign of the Green's function related to a general linear $n^{\rm th}$-order operator, depending on a real parameter, $T_n[M]$, coupled with the $(k,n-k)$ boundary value conditions. 
	
	If operator $T_n[\bar M]$ is disconjugate for a given $\bar M$, we describe the interval of values on the real parameter $M$ for which the Green's function has constant sign. 
	
	One of the extremes of the interval is given by the first eigenvalue of operator $T_n[\bar M]$ satisfying  $(k,n-k)$ conditions. 
	
	The other extreme is related to the minimum (maximum) of the first eigenvalues of $(k-1,n-k+1)$ and $(k+1,n-k-1)$ problems.
	
	Moreover if $n-k$ is even (odd) the Green's function cannot be non-positive (non-negative). 

To illustrate the applicability of the obtained results, we calculate the parameter intervals of constant sign  Green's functions for particular operators. Our method avoids the necessity of calculating the expression of the Green's function.

We finalize the paper by presenting a particular equation in which it is shown that the disconjugation hypothesis on operator $T_n[\bar M]$  for a given $\bar M$ cannot be eliminated.

	\vspace{.5cm}
	
	
\end{abstract}

\section{Introduction}

 It is very well known that the validity of the method of lower and upper solutions, coupled with the monotone iterative techniques \cite{CoHa, llv}, is equivalent to the constant sign of the Green's function related to the linear part of the studied problem \cite{Cab, A2}. Moreover, by means of the celebrated Krasnosel'ski\u{\i} contraction/expansion fixed point theorem \cite{krasnoselskii}, nonexistence, existence and multiplicity results are derived from the construction of suitable cones on Banach spaces. Such construction follows by using adequate properties of the Green's function, one of them is its constant sign \cite{caci, graefkongwangAML, graefkongwang, to}. Recently, the combination of the two previous methods has been proved as a useful tool to ensure the existence of solution \cite{caci2, CCI, jc-df-fm, df-gi-jp-prse,persson}.

Having in mind the power of this constant sign property, we will describe the interval of parameters for which the Green's function related to the general linear $n^{\rm th}$-order equation 
\begin{equation}
\label{e-Ln}
T_n[M]\, u(t) \equiv u^{(n)}(t)+a_1(t)\, u^{(n-1)}(t)+ \cdots +a_{n-1}(t)\, u'(t)+(a_{n}(t)+M)\, u(t)=0
\end{equation}
$ t \in I \equiv [a,b]$,	coupled with the so-called $(k,n-k)$ two point boundary value conditions: 
\begin{equation}
\label{e-k-n-k}
u(a)=u'(a)=\cdots=u^{(k-1)}(a)=u(b)=u'(b)=\cdots=u^{(n-k-1)}(b)= 0,
\end{equation}
$1 \le k\le n-2$,
has constant sign on its square of definition $I \times I$.

The main hypothesis consists on assuming  that there is a real parameter $\bar M$ for which operator $T_n[\bar M]$ is disconjugate on $I$. 

An exhaustive study of the general theory and the fundamental properties of the disconjugacy are compiled  in the classical book of Coppel \cite{Cop}. Different sufficient criteria to ensure the disconjugacy character of the linear operator $T_n[0]$ has been developed in the literature, we refer the classical references \cite{Ze1, Ze2}. Sufficient conditions for particular cases have been obtained in \cite{Erbe,KwZe, Si} and, more recently, in \cite{Eli}. We mention that operator $u^{(n)}(t)+a_{1}(t)\, u^{(n-1)}(t)$ is always disconjugate in $I$, see \cite{Cop} for details, in particular the results here presented are valid for operator $u^{(n)}(t)+M\, u(t)$.

As it has been shown in \cite{Cop}, the disconjugacy character implies the constant sign of the Green's function $g_M$ related to problem \eqref{e-Ln}--\eqref{e-k-n-k}. However, as we will see along the paper, the reciprocal property is not true in general: there are real parameters $M$ for which the Green's function has constant sign but the equation \eqref{e-Ln} is not disconjugate. In other words, the disconjugacy character is only a sufficient condition in order to ensure the constant sign of a Green's function related to problem \eqref{e-Ln}--\eqref{e-k-n-k}.

In fact, from the disconjugacy character of operator $T_n[\bar M]$ in $I$, it is shown in \cite{Cop} that the Green's function $g_M$ satisfies a suitable condition, stronger than its constant sign. Such condition fulfills the one introduced in \cite[Section 1.8]{Cab}. So, following the results given in that reference we conclude that the set of parameters $M$ for which $g_M$ has constant sign is an interval $H_T$. Moreover if $n-k$ is even then the maximum of $H_T$ is the opposed to the biggest negative eigenvalue of problem \eqref{e-Ln}--\eqref{e-k-n-k}, when $n-k$ is odd the minimum of $H_T$ is the opposed to the least positive eigenvalue of such problem.

Thus, the difficulty remains in the characterization of the other extreme of the interval $H_T$. In this case, as it is shown in \cite[Section 1.8]{Cab}, such extreme is not an eigenvalue of the considered problem, so to attain its exact value is not immediate. In practical situations it is necessary to obtain the expression of the Green's function, which is, in general, a difficult matter to deal with. We point out that this problem is not restricted to the $(k,n-k)$ boundary conditions, the difficulty in obtaining the non eigenvalue extreme remains true for any kind of linear conditions \cite{cacisa, lfb}.  In \cite{CCM}, provided operator $T_n[M]$ has constant coefficients, it has been developed a computer algorithm that calculates the exact expression of a Green's function  coupled with two-point boundary value conditions. However, such expression is often  too complicated to manage, and to describe the interval $H_T$ is really very difficult in practical situations. In fact there is not a direct method of construction for non constant coefficients.

We mention that the disconjugacy theory has been used in \cite{ML} to obtain the values for which the third order operators $u'''+M u^{(i)}$, $i=0,1,2$, coupled with conditions $(1,2)$ and $(2,1)$ have constant sign Green's function. Similar procedure has been done in \cite{CE} for the fourth order operator $u^{(4)}+M u$,  coupled with conditions $(2,2)$ and, more recently, in  \cite{CaFe} with conditions $(1,3)$ and $(3,1)$. In all the situations it is obtained the interval of  disconjugacy and then, by means of the expression of the Green's function, it is proved that such interval is optimal. As we have mentioned above, this coincidence holds only in particular cases as the ones treated in these papers, in general the intervals of disconjugacy and constant sign Green's functions do not coincide for the $n^{th}$ - order operator $T_n[M]$.

It is for this that we make in this work a general characterization of the regular extreme of the interval of constant sign $H_T$ by means of the spectral theory. We will show that it is an eigenvalue of the same operator $T_n[M]$ but related to different two-point boundary value conditions. In fact, if $n-k$ is even, it will be the minimum of the two least positive eigenvalues related to conditions  $(k-1,n-k+1)$ and $(k+1,n-k-1)$. It will be the maximum of the two biggest negative eigenvalues of such problems when $n-k$ is odd. So, we make a general characterization for the general operator $T_n[M]$ and we avoid the necessity of calculate the Green's function and to study its sign dependence on the real parameter $M$. 

We note that if operator $T_n[M]$ has constant coefficients, to obtain the corresponding eigenvalues we only must to calculate the determinant of the matrix of coefficients of a linear homogeneous algebraic system. Numerical methods are also valid for the non-constant case.

It is important to mention that, as consequence of the obtained results, denoting by $g_M$ the Green's function related to problem \eqref{e-Ln}--\eqref{e-k-n-k}, we conclude that $(-1)^{n-k} \, g_M(t,s)$ cannot be negative on $I \times I$ for all $M \in \R$.

The paper is scheduled as follows: in a preliminary section 2 we introduce the fundamental concepts that are needed in the development of the paper. Next section is devoted to the proof of the main result in which the regular extreme is obtained via spectral theory.  In  section 4 some particular cases are considered where it is shown the applicability of the obtained results. In last section is introduced an example that shows that the disconjugacy hypothesis on the main result cannot be eliminated.
	
\section{Preliminaries}

In this section, for the convenience of the reader, we introduce the fundamental tools in the theory of disconjugacy and Green's functions that will be used in the development of further sections.

\begin{definition}
	Let $a_k\in C^{n-k}(I)$ for $k=1,\dots,n$. The $n^{\rm th}$-order linear differential equation (\ref{e-Ln}) is said to be disconjugate on an interval $I$ if every non trivial solution has less than $n$ zeros on $I$, multiple zeros being counted according to their multiplicity.
\end{definition}
\begin{definition}
	The functions $u_1,\dots, u_n \in C^n(I)$ are said to form a Markov system on the interval $I$ if the $n$ Wronskians
	\begin{equation}
	W(u_1,\dots,u_k)=\left| \begin{array}{ccc}
	u_1&\cdots&  u_k\\
	\vdots&\cdots&\vdots\\
	u_1^{(k-1)}&\cdots&u_k^{(k-1)}\end{array}\right| \,,\quad k=1,\dots,n \,,
	\end{equation}
	are positive throughout $I$.
\end{definition}

The following result about this concept is collected on \cite[Chapter 3]{Cop}.

\begin{theorem}\label{T::4}
	The linear differential equation (\ref{e-Ln}) has a Markov fundamental system of solutions on the compact interval $I$ if, and only if, it is disconjugate on $I$.
	
\end{theorem}

In order to introduce the concept of Green's function related to the $n^{\rm th}$ - order scalar problem \eqref{e-Ln}-\eqref{e-k-n-k}, we consider the following equivalent first order vectorial problem:

\begin{equation}\label{Ec::vec}
x'(t)=A(t)\, x(t)\,,\ t\in I\,,\quad 
B\,x(a)+C\,x(b)=0,
\end{equation}
with $x(t) \in \R^n$, $A(t), \, B,\,\ C\in \mathcal{M}_{n\times n}$, defined by
\[x(t)=\left( \begin{array}{c} 
u(t)\\
u'(t)\\
\vdots\\
u^{(n-1)}(t) \end{array}\right),\, \quad A(t)=\left( \begin{array}{c|c}
&\\
0&\quad I_{n-1}\quad\\
&\\
\hline
-(a_n(t)+M)&-a_{n-1}(t)\cdots-a_1(t) \end{array}\right), \]

\begin{equation}\label{Ec::Cf}
B=\left( \begin{array}{c|c} I_k&0\\
\hline 0&0\end{array}\right), \;\quad C=\left( \begin{array}{c|c} 0&0\\
\hline I_{n-k}&0\end{array}\right).
\end{equation}

Here $I_j$, $j=1, \ldots,n-1$, is the $j \times j$ identity matrix.

\begin{definition} \label{Def::G}
	We say that $G$ is a Green's function for problem \eqref{Ec::vec} if it satisfies the following properties:
	\begin{itemize}
\item[$\mathrm{(G1)}$] $G\equiv (G_{i,j})_{i,j}\in{1,\dots,n}\colon (I\times I)\backslash \left\lbrace (t,t)\,,\ t\in I\right\rbrace \rightarrow \mathcal{M}_{n\times n}$.

\item[$\mathrm{(G2)}$] $G$ is a $C^{1}$ function on the triangles $\left\lbrace (t,s)\in \mathbb{R}^2\,,\quad a\leq s<t\leq b\right\rbrace $ and $\left\lbrace (t,s)\in \mathbb{R}^2\,,\ a\leq t < s\leq b\right\rbrace $.

\item[$\mathrm{(G3)}$] For all $i\neq j$ the scalar functions $G_{i,j}$ have a continuous extension to $I\times I$.

\item[$\mathrm{(G4)}$] For all $s\in(a,b)$, the following equality holds:
\[\dfrac{\partial }{\partial t}\, G(t,s)=A(t)\,G(t,s)\,,\quad \text{for all } t\in I\backslash \left\lbrace s\right\rbrace \,.\]

\item[$\mathrm{(G5)}$] For all $s\in(a,b)$ and $i\in\left\lbrace 1,\dots, n\right\rbrace $, the following equalities are fulfilled:
\[\lim_{s\rightarrow t^+}G_{i,i}(s,t)=\lim_{s\rightarrow t^-}G_{i,i}(t,s)=1+\lim_{s\rightarrow t^+}G_{i,i}(t,s)=1+\lim_{s\rightarrow t^-}G_{i,i}(s,t)\,.\]

	\item[$\mathrm{(G6)}$] For all $s\in(a,b)$, the function $t\rightarrow G(t,s)$ satisfies the boundary conditions
	\[B\,G(a,s)+C\,G(b,s)=0\,.\]
	\end{itemize}
\end{definition}

It is very well known that Green's function related to this problem follows the following expression \cite[Section 1.4]{Cab}
\begin{equation}\label{Ec:MG} G(t,s)=\left( \begin{array}{ccccc}
g_1(t,s)&g_2(t,s)&\cdots&g_{n-1}(t,s)&g_M(t,s)\\&&&&\\
\dfrac{\partial }{\partial t}\,g_1(t,s)& \dfrac{\partial }{\partial t}\,g_2(t,s)&\cdots&\dfrac{\partial }{\partial t}\,g_{n-1}(t,s)& \dfrac{\partial }{\partial t}\,g_M(t,s)\\
\vdots&\vdots&\cdots&\vdots&\vdots\\
\dfrac{\partial^{n-1} }{\partial t^{n-1}}\,g_1(t,s)&\dfrac{\partial^{n-1} }{\partial t^{n-1}}\,g_2(t,s)&\cdots&\dfrac{\partial^{n-1} }{\partial t^{n-1}}\,g_{n-1}(t,s)&\dfrac{\partial^{n-1}} {\partial t^{n-1}}\,g_M(t,s)\end{array} \right) \,,\end{equation}
where $g_M(t,s)$ is the scalar Green's function related to problem (\ref{e-Ln})-(\ref{e-k-n-k}). 

Using  Definition \ref{Def::G} we can deduce the properties fulfilled by $g_M(t,s)$.
In particular, $g_M\in C^{n-2}(I)$ and it satisfies, as a function of $t$, the two-point boundary value conditions \eqref{e-k-n-k}.

We  also mention a  result which appears on \cite[Chapter 3, Section 6]{Cop} and that connects the disconjugacy and the sign of the Green's function related to the problem (\ref{e-Ln})-(\ref{e-k-n-k}).

\begin{lemma}\label{L::1}
	If the linear differential equation (\ref{e-Ln}) is disconjugate and $g_M(t,s)$ is the Green's function related to the problem (\ref{e-Ln})-(\ref{e-k-n-k}), hence
	\begin{eqnarray}\nonumber
	g_M(t,s)\,p(t)&\geq &0\,,\quad (t,s)\in I \times I\,,
	\\\nonumber&&\\\nonumber
	\dfrac{g_M(t,s)}{p(t)}&>&0\,,\quad (t,s)\in[a,b]\times (a,b)\,.
	\end{eqnarray}
	where $p(t)=(t-a)^k\,(t-b)^{n-k}$.
\end{lemma}

The adjoint of the operator $T_n[M]$, is given by the following expression, see for details \cite[Section 1.4]{Cab} or \cite[Chapter 3, Section 5]{Cop},
\begin{equation}\label{EC::Ad}
T_n^*[M]v (t)\equiv (-1)^n \,v^{(n)}(t)+\sum_{j=1}^{n-1}(-1)^j\,\left(a_{n-j}\,v\right)^{(j)}(t)+(a_n(t)+M)\,v(t)\,,
\end{equation}
and its domain of definition is
{\footnotesize \begin{eqnarray} \label{DA}
 D(T_n^*[M])&=&\left\lbrace v\in C^n(I)\ \mid \sum_{j=1}^{n}\sum_{i=0}^{j-1} (-1)^{j-1-i} (a_{n-j}\,v)^{(j-1-i)}(b)\,u^{(i)}(b)\right. \\\nonumber
& &\left. =\sum_{j=1}^{n}\sum_{i=0}^{j-1} (-1)^{j-1-i} (a_{n-j}\,v)^{(j-1-i)}(a)\,u^{(i)}(a) \ \text{ (with $a_0=1$)}\,, \forall u\in D(T_n[M])\right\rbrace \,.
\end{eqnarray}}

In our case, because of boundary conditions (\ref{e-k-n-k}), we can express the domain of the operator $T_n[M]$, $D(T_n[M])$, as
{ \begin{equation*}
X_k= \left\lbrace u\in C^n(I)\ \mid u(a) =\dots=u^{(k-1)}(a)=u(b)=\dots=u^{(n-k-1)}(b)=0\right\rbrace,
\end{equation*}}
so we can replace expression (\ref{DA}) with

{\scriptsize \begin{eqnarray*}\nonumber
	D(T_n^*[M])&=&\left\lbrace v\in C^n(I)\ \mid \sum_{j=n-k+1}^{n}\sum_{i=n-k}^{j-1} (-1)^{j-1-i} (a_{n-j}\,v)^{(j-1-i)}(b)\,u^{(i)}(b)\right. \\\nonumber
	& &\left. = \sum_{j=k+1}^{n}\sum_{i=k}^{j-1} (-1)^{j-1-i} (a_{n-j}\,v)^{(j-1-i)}(a)\,u^{(i)}(a) \ \text{ (with $a_0=1$)}\,,\  \forall u\in C^n(I)\right\rbrace \,.\\&&
	\end{eqnarray*}}

In order to simplify the previous expression, we choose a function $u\in C^n(I)$ satisfying
\begin{eqnarray}\nonumber u^{(\sigma)}(a)&=&0\,,\ \sigma =1,\dots,n-1\,,\\
\nonumber u^{(\mu)}(b)&=&0\,,\ \mu=1,\dots,n-2\,,\\\nonumber
u^{(n-1)}(b)&=&1\,.\end{eqnarray}

Realizing that $a_0=1$, we conclude that every function $v\in D(T_n^*[M])$ must satisfy $v(b)=0$.

Moreover, if we now choose a function in $C^n(I)$ that satisfies
\begin{eqnarray}\nonumber u^{(\sigma)}(a)&=&0\,,\ \sigma =1,\dots,n-1\,,\\
\nonumber u^{(\mu)}(b)&=&0\,,\ \mu=1,\dots,n-1\,,\quad \mu\neq n-2\\\nonumber
u^{(n-2)}(b)&=&1,\,\end{eqnarray}
we conclude that any function $v\in D(T_n^*[M])$ has to satisfy
\[-v'(b)+a_1(b)\,v(b)=0.\] 
Since  $a_1\in C^{n-1}(I)$ and  $v(b)=0$,  we conclude that $v'(b)=0$.

Repeating this process we achieve that the domain of the adjoint operator is given by
\begin{equation}
D(T_n^*[M])=X_{n-k}\,.
\end{equation}

The next result appears in \cite[Chapter 3, Theorem 9]{Cop}
\begin{theorem}\label{T::2}
	The equation (\ref{e-Ln}) is disconjugate on an interval $I$ if, and only if, the adjoint equation, $T_n^*[M]\,y(t)=0$ is disconjugate on $I$.
\end{theorem}

We denote $g_M^*(t,s)$ as the Green function of the adjoint operator, $T_n^*[M]$.

In \cite[Section 1.4]{Cab} it is proved the following relationship
\begin{equation} \label{Ec::gg}
g^*_M(t,s)=g_M(s,t)\,.
\end{equation}

Defining now the following operator
\begin{equation}\label{Ec::Tg}
\widehat{T}_n[(-1)^n\,M]:=(-1)^n T_n^*[M] \,,
\end{equation}
we deduce, from the previous expression, that
\begin{equation}\label{Ec::gg1}
\widehat{g}_{(-1)^n\,M}(t,s)=(-1)^n\,g_M^*(t,s)=(-1)^n\,g_{M}(s,t)\,.
\end{equation}

Obviously, Theorem \ref{T::2} remains true for operator $\widehat{T}_n[(-1)^n\,M]$.

\begin{definition}
\label{d-IP}
Operator $T_n[M]$ is said to be inverse positive (inverse negative) on $X_k$ if every function $u \in X_k$ such that $T_n[M]\, u \ge 0$ in $I$, must verify $u\geq 0$ ($u\leq 0$) on $I$.
\end{definition}

Next results are proved in \cite[Section 1.6, Section 1.8]{Cab}.

\begin{theorem}\label{T::in1}
	Operator $T_n[M]$ is inverse positive (inverse negative) on $X_k$ if, and only if, Green's function related to problem (\ref{e-Ln})-(\ref{e-k-n-k}) is non-negative (non-positive) on its square of definition.
\end{theorem}

\begin{theorem}\label{T::d1}
	Let $M_1$, $M_2\in\mathbb{R}$ and suppose that operators $T_n[M_j]$, $j=1,2$, are invertible in $X_k$.
	Let $g_j$, $j=1,2$, be Green's functions related to  operators $T_n[M_j]$ and suppose that both functions have the same constant sign on $I \times I$. Then, if $M_1<M_2$, it is satisfied that $g_2\leq g_1$ on $I \times I$.
\end{theorem}

In the sequel, we introduce two conditions on $g_M(t,s)$ that will be used along the paper.
\begin{itemize}
	\item[$(P_g$)] Suppose that there is a continuous function $\phi(t)>0$ for all $t\in (a,b)$ and $k_1,\ k_2\in \mathcal{L}^1(I)$, such that $0<k_1(s)<k_2(s)$ for a.e. $s\in I$, satisfying
	\[\phi(t)\,k_1(s)\leq g_M(t,s)\leq \phi(t)\, k_2(s)\,,\quad \text{for a. e. } (t,s)\in I \times I \,.\]
	\item[($N_g$)] Suppose that there is a continuous function $\phi(t)>0$ for all $t\in (a,b)$ and $k_1,\ k_2\in \mathcal{L}^1(I)$, such that $k_1(s)<k_2(s)<0$ for a.e. $s\in I$, satisfying
	\[\phi(t)\,k_1(s)\leq g_M(t,s)\leq \phi(t)\, k_2(s)\,,\quad \text{for a. e. }(t,s)\in I \times I\,.\]
\end{itemize}

Finally, we introduce the following sets, which are going to particularize $H_T$,
\begin{eqnarray*}
P_T&=& \left\lbrace M\in \mathbb{R}\,,\ \mid \quad g_M(t,s)\geq 0\quad \forall (t,s)\in I\times I\right\rbrace, \\
N_T&=& \left\lbrace M\in \mathbb{R}\,,\ \mid \quad g_M(t,s)\leq 0\quad \forall (t,s)\in I\times I\right\rbrace.
\end{eqnarray*}

Next results describe the structure of the two previous parameter's set.
\begin{theorem}\label{T::6} \cite[Lemma 1.8.33]{Cab}
	Let $\bar{M}\in\mathbb{R}$ be fixed. Suppose that operator $T_n[\bar{M}]$ is invertible on $X_k$, its related Green's function is non-negative on $I \times I$, it satisfies condition ($P_g$), and the set $P_T$ is bounded from above.
	Then $P_T=(\bar{M}-\lambda_1,\bar{M}-\bar{\mu}]$, with $\lambda_1>0$ the least positive eigenvalue of operator $T_n[\bar{M}]$ in $X_k$ and $\bar{\mu} \le 0$ such that $T_n[\bar{M}-\bar{\mu}]$ is invertible in $X_k$ and the related non-negative Green's function $g_{\bar{M}-\bar{\mu}}$ vanishes at some points on the square $I\times I$.
\end{theorem}

\begin{theorem}\label{T::7} \cite[Lemma 1.8.25]{Cab}
	Let $\bar{M}\in\mathbb{R}$ be fixed. Suppose that operator $T_n[\bar{M}]$ is invertible in $X_k$, its related Green's function is non-positive on $I\times I$, it satisfies condition ($N_g$), and the set $N_T$ is bounded from below.
	Then $N_T=[\bar{M}-\bar{\mu},\bar{M}-\lambda_1)$, with $\lambda_1<0$ the biggest negative eigenvalue of operator $T_n[\bar{M}]$ in $X_k$ and $\bar{\mu}\ge 0$ such that $T_n[\bar{M}-\bar{\mu}]$ is invertible in $X_k$ and the related non-positive Green's function $g_{\bar{M}-\bar{\mu}}$ vanishes at some points on the square $I\times I$.
\end{theorem}

\section{Main Result}

This section is devoted to prove the eigenvalue characterization of the sets $P_T$ and $N_T$. Such result is enunciated on the following Theorem

\begin{theorem}\label{T::5}
	Let $\bar{M}\in\mathbb{R}$ be such that equation $T_n[\bar{M}]\,u(t)=0$  is disconjugate on $I$. Then the two following properties are fulfilled:\\
	
If $n-k$ is even then the operator $T_n[M]$  is inverse positive on $X_k$ if, and only if, $M\in (\bar{M}-\lambda_1,\bar{M}-\lambda_2]$, where:
	\begin{itemize}
\item $\lambda_1>0$ is the least positive eigenvalue  of operator $T_n[\bar{M}]$ in $X_k$.
\item $\lambda_2<0$ is the maximum of:
\begin{itemize}
	\item $\lambda_2'<0$, the biggest negative eigenvalue  of operator $T_n[\bar{M}]$ in $X_{k-1}$.
	\item $\lambda_2''<0$, the biggest negative eigenvalue of operator $T_n[\bar{M}]$ in $X_{k+1}$.\\
\end{itemize}
	\end{itemize}

 If $n-k$ is odd then the operator $T_n[M]$  is inverse negative on $X_k$ if, and only if, $M\in [\bar{M}-\lambda_2,\bar{M}-\lambda_1)$, where:
\begin{itemize}
	\item $\lambda_1<0$ is the biggest negative eigenvalue of operator $T_n[\bar{M}]$ in $X_k$.
	\item $\lambda_2>0$ is the minimum of:
	\begin{itemize}
\item $\lambda_2'>0$, the least positive eigenvalue  of operator $T_n[\bar{M}]$ in $X_{k-1}$.
\item $\lambda_2''>0$, the least positive eigenvalue  of operator $T_n[\bar{M}]$ in $X_{k+1}$.
	\end{itemize}
\end{itemize}

\end{theorem}

In order to prove this result, we separate the proof in several subsections.
 \subsection{Decomposition of operator $T_n[\bar{M}]$}\label{SC::1}
 
 We are interested into put operator $T_n[\bar{M}]$ as a composition of suitable operators of order $h \le n$. Such expression allow us to control the values of such operators at the extremes of the interval $a$ and $b$.  
 
We recall the following result proved in \cite[Chapter 3]{Cop}
 
\begin{theorem}\label{T::3}
	The linear differential equation (\ref{e-Ln}) has a Markov system of solutions if, and only if, the operator $T_n[M]$ has a representation
	\begin{equation}
	\label{e-descomp}
	T_n[M]\,y\equiv v_1 \,v_2\,\dots\,v_n \dfrac{d}{dt}\left( \dfrac{1}{v_n}\,\dfrac{d}{dt}\left( \cdots \dfrac{d}{dt}\left( \dfrac{1}{v_2} \dfrac{d}{dt}\left( \dfrac{1}{v_1}\,y\right) \right) \right) \right) \,,	
	\end{equation}
	where $v_k>0$  on $I$ and $v_k\in C^{n-k+1}(I)$ for $k=1,\dots,n$.
\end{theorem}

It is obvious that for any real parameter $M$, denoting $\lambda=M-\bar{M}$, we can rewrite operator $T_n[M]$ as follows:
	
	\[T_n[M] \,u(t)\equiv  T_n[\bar{M}]\,u(t)+\lambda\,u(t).\]
	
	If we assume that equation $T_n[\bar{M}] u(t)=0$ is disconjugate on $I$, because of Theorems \ref{T::4} and \ref{T::3}, we can express $T_n[\bar{M}]$ as
	
	\[T_n[\bar{M}]\,u(t)\equiv v_1(t)\,\dots\,v_n(t)\,T_n u(t)\,,\]
	where $T_k$ are built as
	\begin{equation}
	\label{e-Tk}
	T_0 u(t)=u(t)\,,\quad T_k u(t)=\dfrac{d}{dt}\left( \dfrac{1}{v_k(t)}\,T_{k-1} u(t)\right),\;  k =1, \ldots,n-1, \; t\in I,
	\end{equation}
	with $v_k>0$ on $I$, $v_k\in C^{n-k+1}(I)$, for $k=1,\dots,n$.\\
	
	Let us see now that $T_h\,u(t)$ is given as a linear combination of $u(t), u'(t),\dots, u^{(h)}(t)$ with the form
	\begin{equation}\label{Ec::7}
	T_h\,u(t)=\dfrac{1}{v_1(t)\dots v_h(t)}\,u^{(h)}(t)+p_{h_1}(t)\,u^{(h-1)}(t)+\cdots p_{h_h}(t)\,u(t)\,,
	\end{equation}
	where $p_{h_i}\in C^{n-h}(I)$.
	
	Indeed, we are going to prove this equality by induction.
	
	For $h=1$
	\begin{equation*}
	T_1\,u(t)=\dfrac{d}{dt}\left( \dfrac{1}{v_1(t)}\,u(t)\right) =\dfrac{1}{v_1(t)}\,u'(t)-\dfrac{v_1'(t)}{v_1^2(t)}\,u(t)\,.
	\end{equation*}
	
	Assume, by induction hypothesis, that equation (\ref{Ec::7}) is satisfied for some $h\in\left\lbrace 1,\dots,n-1\right\rbrace$ , therefore
{\footnotesize 	\begin{eqnarray}\nonumber
	T_{h+1}\,u(t)&=&\dfrac{d}{dt}\left( \dfrac{1}{v_{h+1}(t)}\left( \dfrac{1}{v_1(t)\dots v_h(t)}\,u^{(h)}(t)+p_{h_1}(t)\,u^{(h-1)}(t)+\cdots p_{h_h}(t)\,u(t)\right) \right) \\\nonumber\\\nonumber
	&=&\dfrac{d}{dt}\left(  \dfrac{1}{v_1(t)\dots v_{h+1}(t)}\,u^{(h)}(t)\right) +\dfrac{d}{dt}\left( \dfrac{1}{v_{h+1}(t)}\left(p_{h_1}(t)\,u^{(h-1)}(t)+\cdots p_{h_h}(t)\,u(t)\right) \right)\,,
	\end{eqnarray}}which, clearly has the form of equation (\ref{Ec::7}).
	
Finally, taking into account boundary conditions (\ref{e-k-n-k})  and the regularity of functions $p_{h_i}$, we conclude that 
	\[T_0u(a)=0\,,\dots,T_{k-1}u(a)=0\,,\ T_0u(b)=0\,,\dots, T_{n-k-1}(b)=0.\]

Moreover	
	\begin{eqnarray}
	\label{e-Tk-a}
T_k u(a)&=& \dfrac{1}{v_1(a)\dots v_k(a)}\,u^{(k)}(a)\,,\\\nonumber\\
\label{e-Tk-b}
T_{n-k} u(b) &=& \dfrac{1}{v_1(b)\dots v_{n-k}(b)}\,u^{(n-k)}(b)\,.
	\end{eqnarray}

	So, from the positiveness of $v_h$ on $I$, $h \in \{1, \dots,n\}$, we have that $T_k\,u(a)$ and $u^{(k)}(a)$ have the same sign. The same property holds for $T_{n-k}\,u(b)$ and $u^{(n-k)}(b)$.

\subsection{Expression of the matrix Green's function}

This subsection is devoted to express, as functions of $g_M(t,s)$, the functions $g_1(t,s), \dots, g_{n-1}(t,s)$, defined on \eqref{Ec:MG}, as the first row componentes of the Green's function of the vectorial system \eqref{Ec::vec}.

By studying the adjoint operator as in \cite[Section 1.3]{Cab}, we know that the related Green's function of the adjoint operator $G^*$ satisfies that $G^*(t,s)=G^T (s,t)$. Moreover, the following equality holds:

\begin{equation*}\label{Ec::3}\dfrac{\partial }{\partial t}\,\left( -G^*(t,s)\right) =-A^T (t)\, \left( -G^*(t,s)\right) \,, \quad t \in  I\backslash \left\lbrace s\right\rbrace.
\end{equation*}

So, we can transform previous equality in
{\footnotesize \begin{equation*}
\left( -\dfrac{\partial }{\partial t}\,G(s,t)\right) ^T=-\dfrac{\partial }{\partial t}\,G^T(s,t)=-A^T(t) \left( -G^T(s,t)\right)=A^T(t)\,G^T(s,t)=\left( G(s,t)\, A(t)\right)^T\,.
\end{equation*}}

Hence
\[\dfrac{\partial }{\partial t}\,G(s,t)= -G(s,t)\,A(t)\,,\]
or, which is the same,
\begin{equation}\label{Ec::4}\dfrac{\partial }{\partial s}\,G(t,s)= -G(t,s)\,A(s)\,.\end{equation}

Using this equality, we are going to prove by induction the following ones
\begin{equation}\label{Ec::6}
g_{n-j}(t,s)=(-1)^j\,\dfrac{\partial ^j }{\partial s^j}\,g_M(t,s)+\sum_{k=0}^{j-1} \alpha_k^j (s)\,\dfrac{\partial ^k }{\partial s^k}\,g_M(t,s),\;\; j=1, \ldots,n-1.
\end{equation}

Here $\alpha_i^j(s)$ are functions of $a_1(s)\,,\dots,a_j(s)$ and of its derivatives until order $(j-1)$ and  follow the recurrence formula

\begin{eqnarray}
\label{r2}
\alpha_k^{j+1}(s)&=&0\,,\quad k\geq j+1,\\
\label{r1}
\alpha_0^{j+1}(s)&=&a_{j+1}(s)-\left( \alpha_0^j\right) '(s), \quad j \ge 1,\\\label{r4}
\alpha_k^{j+1}(s)&=&-\left( \alpha_{k-1}^j(s)+\left( \alpha_k^j\right) '(s)\right), \quad 1\leq k\leq j.
\end{eqnarray}

Using equality (\ref{Ec::4}), we deduce that the Green's matrix' terms which are on position $(1,i)$, $i=1,\dots,n$, satisfy the following equality

\begin{equation}\label{Ec::5}
g_{i-1}(t,s)=-\dfrac{\partial }{\partial s}\,g_i(t,s)+a_{n-i+1}(s)\,g_M(t,s)\,,\quad i=2,\dots,n\,,
\end{equation}
where $g_M(t,s)\equiv g_n(t,s)$.

If we take $i=n$ in equation (\ref{Ec::5}) we deduce
\begin{equation*}
g_{n-1}(t,s)=-\dfrac{\partial }{\partial s}\,g_M(t,s)+a_1(s)\,g_M(t,s),
\end{equation*}
which give us  equation (\ref{Ec::6}) for $j=1$.

Assume now that equalities \eqref{Ec::6} -- \eqref{r4} are fulfilled for some $j\in\{1,\dots,n-2\}$ given. Let us see that they hold again for $j+1$.

\begin{eqnarray}
\nonumber
g_{n-j-1}(t,s)&=& -\dfrac{\partial}{\partial s}\left( (-1)^j \dfrac{\partial ^j }{\partial s^j}\,g_M(t,s)+\sum_{k=0}^{j-1}\alpha_k^j(s)\dfrac{\partial^k }{\partial s^k}\,g_M(t,s)\right) \\
\nonumber && +a_{j+1}(s)\,g_M(t,s)\\
\nonumber&=&a_{j+1}(s)\,g_M(t,s)+(-1)^{j+1} \dfrac{\partial^{j+1} }{\partial s^{j+1}}\,g_M(t,s)\\
\nonumber
&& -\sum_{k=0}^{j-1}\left( \alpha_k^j\right) '(s)\dfrac{\partial ^k }{\partial s^k}\,g_M(t,s) -\sum_{k=0}^{j-1}\alpha_k^j(s)\dfrac{\partial ^{k+1} }{\partial s^{k+1}}\,g_M(t,s)\\\nonumber
&=&(-1)^{j+1} \dfrac{\partial^{j+1} }{\partial s^{j+1}}\,g_M(t,s)+a_{j+1}(s)\,g_M(t,s)
\\
\nonumber
&&
-\sum_{k=0}^{j-1}\left( \alpha_k^j\right) '(s)\dfrac{\partial ^k}{\partial s^k}\, g_M(t,s) -\sum_{k=1}^{j}\alpha_{k-1}^j(s)\dfrac{\partial ^{k} }{\partial s^{k}}\,g_M(t,s)
\\\nonumber
&=& (-1)^{j+1} \dfrac{\partial^{j+1} }{\partial s^{j+1}}\,g_M(t,s)+\sum_{k=0}^j \alpha_k^{j+1}(s)\,\dfrac{\partial^k }{\partial s^k}\,g_M(t,s)\,.
\end{eqnarray}

Now, we can express Green's matrix related to problem (\ref{Ec::vec}), $G(t,s)$, as

\begin{equation}\label{MG1}
\scriptsize{\left( \begin{array}{ccc}
	(-1)^{n-1}\,\dfrac{\partial ^{n-1} }{\partial s^{n-1}}\,g_M(t,s)+\displaystyle \sum_{k=0}^{n-2}\alpha_k^{n-1}(s)\,\dfrac{\partial^k}{\partial s^k}\,g_M(t,s)&\cdots&g_M(t,s)\\&&\\
	(-1)^{n-1}\,\dfrac{\partial ^{n} }{\partial t\,\partial s^{n-1}}\,g_M(t,s)+\displaystyle \sum_{k=0}^{n-2}\alpha_k^{n-1}(s)\,\dfrac{\partial^{k+1}}{\partial t\,\partial s^k}\,g_M(t,s)&\cdots&\dfrac{\partial }{\partial t}\,g_M(t,s)\\&\cdots&\\
	\vdots&&\vdots\\&\cdots&\\
	(-1)^n\dfrac{\partial ^{2n-2} }{\partial t^{n-1}\partial s^{n-1}}\,g_M(t,s)+\displaystyle \sum_{k=0}^{n-2}\alpha_k^{n-1}(s)\dfrac{\partial^{n-1+k}}{\partial t^{n-1}\partial s^k}\,g_M(t,s)&\cdots&\dfrac{\partial^{n-1} }{\partial t^{n-1}}\,g_M(t,s)\end{array}\right) }\end{equation}

If coefficients $a_1(s),\dots,a_{n-1}(s),\, a_n(s)$ are constants, $a_1,\dots,a_{n-1},a_n$, we can solve explicitly the recurrence form \eqref{r2} -- \eqref{r4} and deduce that 
$$\alpha_k^j(s)=(-1)^k\,a_{j-k}.$$ 

So, we have that
\[g_{n-j}(t,s)=\sum_{k=0}^j(-1)^k\,a_{j-k}\,\dfrac{\partial^k}{\partial s^k}\, g_M(t,s)\,,\quad \text {with } a_0=1\,,\]
and we can rewrite $G(t,s)$ as
\[
\tiny{\left( \begin{array}{cccc}
	\displaystyle\sum_{k=0}^{n-1}(-1)^k\,a_{n-1-k}\,\dfrac{\partial^k}{\partial s^k}\,g_M(t,s)&\cdots&\displaystyle\sum_{k=0}^{1}(-1)^k\,a_{1-k}\,\dfrac{\partial^k}{\partial s^k}\,g_M(t,s)&g_M(t,s)\\&&&\\
	\displaystyle\sum_{k=0}^{n-1}(-1)^k\,a_{n-1-k}\,\dfrac{\partial^{k+1}}{\partial t\partial s^k}\,g_M(t,s)&\cdots& \displaystyle\sum_{k=0}^{1}(-1)^k\,a_{1-k}\,\dfrac{\partial^{k+1}}{\partial t\partial s^{k}}\,g_M(t,s)&\dfrac{\partial }{\partial t}\,g_M(t,s)\\&&&\\
	\vdots&&\vdots&\\&&&\\
	\displaystyle\sum_{i=0}^{n-1}(-1)^k\,a_{n-1-k}\,\dfrac{\partial^{n-1+k}}{\partial t^{n-1}\partial s^{k}}\,g_M(t,s)&\cdots&\displaystyle\sum_{k=0}^{1}(-1)^k\,a_{1-k}\,\dfrac{\partial^{n-1+k}}{\partial t^{n-1}\partial s^k}\,g_M(t,s)&\dfrac{\partial^{n-1} }{\partial t^{n-1}}\,g_M(t,s)\end{array}\right) }\,.\]

In particular, if $T_n[M]\,u(t)\equiv u^{(n)}(t)+M\,u(t)$ we conclude that 

\[g_{n-j}(t,s)=(-1)^j\,\dfrac{\partial ^j }{\partial s^j}\,g_M(t,s)\,,\]
so Green's matrix, $G(t,s)$, is given by expression

\[
\left( \begin{array}{cccc}
(-1)^{n-1}\,\dfrac{\partial^{n-1}}{\partial s^{n-1}}\,g_M(t,s)&\cdots&-\dfrac{\partial }{\partial s}\,g_M(t,s)&g_M(t,s)\\&&&\\
(-1)^{n-1}\,\dfrac{\partial^{n}}{\partial t\partial s^{n-1}}\,g_M(t,s)&\cdots& -\dfrac{\partial^2}{\partial t\partial s}\,g_M(t,s)&\dfrac{\partial }{\partial t}\,g_M(t,s)\\&&&\\
\vdots&&\vdots&\\&&&\\
(-1)^{n-1}\dfrac{\partial^{2\,n-2}}{\partial t^{n-1}\partial s^{n-1}}\,g_M(t,s)&\cdots&-\dfrac{\partial^{n}}{\partial t^{n-1}\partial s}\,g_M(t,s)&\dfrac{\partial^{n-1} }{\partial t^{n-1}}\,g_M(t,s)\end{array}\right)\,.\]

\begin{remark}
We note that in the general case it is possible to obtain some of the components of system \eqref{r2} -- \eqref{r4}.
\begin{eqnarray}
\nonumber
\alpha_0^j(s)&=&\sum_{i=0}^{j-1}(-1)^{i}\,a_{j-i}^{(i)}(s)\,,\\\nonumber\alpha_1^j(s)&=&\sum_{i=1}^{j-1}(-1)^{i}\,i\,a_{j-i}^{(i-1)}(s)\,,\\\nonumber \alpha_j^{j+1}(s)&=& (-1)^{j+1}\,a_1(s)\,. \end{eqnarray}
\end{remark}

\subsection{Proof of the main results}

Now we will proceed with the proof of the Main Theorem \ref{T::5}. To this end, we will divide the proof in several steps. 

First, we are going to show a previous lemma.
 
\begin{lemma}\label{L::2}
Let $\bar{M}\in\mathbb{R}$, such that $T_n[\bar{M}]\,u(t)=0$ is disconjugate on $I$. Then the following properties are fulfilled:
	\begin{itemize}
\item If $n-k$ is even, then $T_n[\bar{M}]$ is a inverse positive operator on $X_k$ and its related Green's function, $g_{\bar{M}}(t,s)$, satisfies ($P_g$).

\item If $n-k$ is odd, then $T_n[\bar{M}]$ is a inverse negative operator on $X_k$  and its related Green's function satisfies ($N_g$).

	\end{itemize}
\end{lemma}
\begin{proof}
By Lemma \ref{L::1} we have that for all $s\in(a,b)$
	\begin{eqnarray}\nonumber \exists \lim_{t\rightarrow a^+} \dfrac{g_{\bar{M}} (t,s) }{p(t)}&=&\ell_1(s) >0,\\\nonumber\\
	\nonumber \exists \lim_{t\rightarrow b^-} \dfrac{g_{\bar{M}} (t,s) }{p(t)}&=&\ell_2(s) > 0,
	\end{eqnarray}
	so, for each $s\in(a,b)$, we have that function $
\dfrac{g_{\bar{M}}(t,s)}{p(t)}$ is a strictly positive and continuous function in $I$, thus
\begin{equation}\label{Ec::pg}
0<k_1(s)=\min_{t\in I} \dfrac{g_{\bar{M}}(t,s)}{p(t)}< \max_{t\in I} \dfrac{g_{\bar{M}}(t,s)}{p(t)}=k_2(s)\,,\quad s\in (a,b)\,.  \end{equation}

Since $g_M$ is a continuous function, we have that $k_1$ and $k_2$ are continuous functions too.

If $n-k$ is even, we take $\phi(t)=p(t)$ and condition $(P_g)$ is trivially fulfilled.

If $n-k$ is odd, we take $\phi(t)=-p(t)$ and multiplying equation \eqref{Ec::pg} by $-1$, condition $(N_g)$ holds immediately.	
\end{proof}

	First, notice that, as a direct corollary of the previous Lemma the assertion for $\lambda_1$ in Theorem \ref{T::5} follows directly from Theorems \ref{T::6} and \ref{T::7}.\\
	
	Now, we are going to prove the assertion in Theorem \ref{T::5} concerning $\lambda_2$. 
	
	The proof will be done in several steps. In a first moment we will show that, if $n-k$ is even, the Green's function changes sign for all $M>\bar{M}-\lambda_2$ and for all $M<\bar{M}-\lambda_2$ when $n-k$ is odd.
	
	After that we will prove that such estimation is optimal in both situations.
	
	In order to make the paper more readable, along the proofs of this subsection it will be assumed that $n-k$ is even.
	The arguments with $n-k$ odd will be pointed out at the end of the subsection.
	
\begin{itemize}
	\item [Step 1.] \textit{Behavior of Green's function on a neighborhood of $s=a$ and $s=b$.}
\end{itemize}

First, we construct two functions that will characterize the values of $M\in\mathbb{R}$ for which Green's function  oscillates, or not, on a neighborhood of $s=a$ and $s=b$.

	In order to do that, we denote Green's function related to problem (\ref{e-Ln})-(\ref{e-k-n-k}) as follows
		
	\[g_M(t,s)=\left\lbrace \begin{array} {lcr}
	g_M^1(t,s),&&a\leq t<s\leq b,\\&&\\
	g_M^2(t,s),&&a\leq s\leq t\leq b.\end{array}\right. \,\]
	
	Since $g_M(t,s)$ is a Green's function, it is satisfied that
	\[T_n[M]\,g_M(t,s)=0\,,\quad t\in[a,b]\,,\quad t\neq s\,,\]
	where $g_M(t,s)$ is acting as a function of $t$. 
	
	Therefore, differentiating the previous expression, we deduce that
	\begin{equation}\label{Ec::8}T_n[M] \left( \dfrac{\partial^h g_M(t,s)}{\partial s^h}\right) =\dfrac{\partial^h}{\partial s^h}\left( T_n[M]\,g_M(t,s)\right) =0,\; h=0,\dots,n-1,\, t\neq s\,.\end{equation}
	
	In particular,  we can define the functions
	\begin{eqnarray}
	\label{d-u}
	u(t)&=&\dfrac{\partial ^k}{\partial s^k}g_M^1(t,s)_{\mid_{s=b}}\equiv {g_M^1}_{s^k}(t,b)\,,\quad t\in I\,,\\\nonumber&&\\
	\label{d-v}
	v(t)&=&\dfrac{\partial ^{n-k}}{\partial s^{n-k}}g_M^2(t,s)_{\mid_{s=a}}\equiv {g_M^2}_{s^{n-k}}(t,a)\,,\quad t\in I\,.
	\end{eqnarray}
	
	Because of the relation between $g_M(t,s)$ and $g^*_M(t,s)$, shown in (\ref{Ec::gg}), and taking into account the boundary conditions of the adjoint operator, it is not difficult to deduce that 
	\begin{eqnarray}\nonumber{g_M^2}_{s^{h}}(t,a)&=&{g^{*\,1}_M}_{t^{h}}(a,s)=0,\quad 0\leq h\leq n-k-1,\\\nonumber{g_M^1}_{s^\ell}(t,b)&=&{g^{*\,2}_M}_{t^{\ell}}(b,s)=0\,,\quad 0\leq\ell\leq k-1\,.\end{eqnarray}
	 
	 So, we are interested in to know the values of $M$ for which functions $u(t)$ and $v(t)$ oscillate on $I$. Such property guarantees that Green's function oscillates on a neighborhood of  $s=a$ or $s=b$ for such values. Moreover it provides a higher bound  for the set of parameters where Green's function does not oscillate.
	
	\begin{itemize}
\item[Step 1.1.] \textit{Boundary conditions of $v(t)$.}
	\end{itemize}
	Because of equality (\ref{Ec::8}) we know that $T_n[M] v(t)=0$ on $I$. In this step we are going to see which boundary conditions satisfies function $v$.
	
	We have that $G(t,s)$ as it appears on (\ref{MG1}) is Green's matrix related to vectorial problem (\ref{Ec::vec}). 
	Using the expressions of matrices $B$ and $C$  given by (\ref{Ec::Cf}), if we consider first row of resultant matrix, we obtain for $s\in(a,b)$ the following expression

	\[\left\lbrace \begin{array}{l}
	g_M^1(a,s)=0,\\\\
	-{g_M^1}_s(a,s)+\alpha_0^1(s)\,g_M^1(a,s)=0,\\\\
	\hspace{1cm}\vdots\\\\
	(-1)^{n-k}{g_M^1}_{s^{n-k}}(a,s)+\displaystyle \sum_{i=0}^{n-k-1} \alpha_i^{n-k}(s){g_M^1}_{s^i}(a,s)=0 \,.\end{array} \right.\]
	
	Thus, while $k>1$, none of the previous elements belongs to the diagonal of the matrix Green's function. Since it has discontinuities only at its diagonal entries, see Definition \ref{Def::G}, by considering the limit of $s$ to $a$, we deduce that the previous equalities hold for $g_M^2(a,a)$, i.e. :
	
	\[\left\lbrace \begin{array}{l}
	g_M^2(a,a)=0,\\\\
	-{g_M^2}_s(a,a)+\alpha_0^1(a)\,g_M^2(a,a)=0,\\\\
	\hspace{1cm}\vdots\\\\
	(-1)^{n-k}{g_M^2}_{s^{n-k}}(a,a)+\displaystyle \sum_{i=0}^{n-k-1} \alpha_i^{n-k}(a){g_M^2}_{s^i}(a,a)=0 \,,\end{array} \right.\]
	so, we conclude that
	\[g_M^2(a,a)={g_M^2}_s(a,a)=\cdots={g_M^2}_{s^{n-k}}(a,a)=0\,,\]
hence $v(a)=0$.

	Analogously, since we do not reach any diagonal element, we deduce that $v'(a)=\cdots=v^{(k-2)}(a)=0$. 
	
	Let us see what happens for  $v^{(k-1)}(a)$ with $k>1$. We arrive at the following system written as a function of $g_M^1(t,s)$:

	\[\left\lbrace \begin{array}{l}
	 {g_M^1}_{t^{k-1}}(a,s)=0,\\\\
	- {g_M^1}_{t^{k-1}\,s}(a,s)+\alpha_0^1(s)\,{g_M^1}_{t^{k-1}}(a,s)=0,\\\\
	\hspace{1cm}\vdots\\\\
	(-1)^{n-k} {g_M^1}_{t^{k-1}\,s^{n-k}}(a,s)+\displaystyle \sum_{i=0}^{n-k-1} \alpha_i^{n-k}(s){g_M^1}_{t^{k-1}\,s^i}(a,s)=0 \,.\end{array} \right.\]
	
	This system remains true for $s=a$, and because of the continuity of Green's matrix at  $t=s$ on the non-diagonal elements and the break which is  produced on its diagonal,  we arrive at the following system for   $g_M^2(a,a)$:
	\[\left\lbrace \begin{array}{l}
{g_M^2}_{t^{k-1}}(a,a)=0,\\\\
	-{g_M^2}_{t^{k-1}\,s}(a,a)+\alpha_0^1(a)\,{g_M^2}_{t^{k-1}}(a,a)=0,\\\\
	\hspace{1cm}\vdots\\\\
	(-1)^{n-k} {g_M^2}_{t^{k-1}\,s^{n-k}}(a,a)+\displaystyle \sum_{i=0}^{n-k-1} \alpha_i^{n-k}(a){g_M^2}_{t^{k-1}\,s^i}(a,a)=1,\end{array} \right.\]
hence 
	\[{g_M^2}_{t^{k-1}}(a,a)=\cdots={g_M^2}_{t^{k-1}\,s^{n-k-1}}(a,a)=0\,,\]
	and
	\[v^{(k-1)}(a)= {g_M^2}_{t^{k-1}\,s^{n-k}}(a,a)=(-1)^{n-k}\,.\]
	
	Obviously, taking $k=1$, the same argument tell us that $v(a)=(-1)^{n-1}$.\\

To see the boundary conditions at $t=b$, we have the following system for $s\in(a,b)$, written as a function of $g_M^2(t,s)$
	
	\[\left\lbrace \begin{array}{l}
	g_M^2(b,s)=0,\\\\
	-{g_M^2}_s(b,s)+\alpha_0^1(s)\,g_M^2(b,s)=0,\\\\
	\hspace{1cm}\vdots\\\\
	(-1)^{n-k} {g_M^2}_{s^{n-k}}(b,s)+\displaystyle \sum_{i=0}^{n-k-1} \alpha_i^{n-k}(s) {g_M^2}_{s^i}(b,s)=0,\end{array} \right.\]
	hence
	\[g_M^2(b,s)=\cdots= {g_M^2}_{s^{n-k}}(b,s)=0\,.\]

By continuity, this is satisfied at $s=a$, so 
	\[v(b)= {g_M^2}_{s^{n-k}}(b,a)=0\,.\]
	
Using \eqref{MG1} and \eqref{Ec::Cf}, since there is no jump in this case, it is immediate to verify that $v'(b)=\cdots=v^{(n-k-1)}(b)=0$.
	
As consequence $v$ is the unique solution of the following problem, which we denote as $(P_v)$:
\begin{eqnarray}
\nonumber T_n[M]\,v(t)&=&0\,,\quad t\in I\,,\\\nonumber
v(a)&=&\dots = v^{(k-2)}(a)=0,\\\nonumber
v(b)&=&\dots = v^{(n-k-1)}(b)=0\,,\\\nonumber
 v^{(k-1)}(a)&=&(-1)^{n-k}\,.
 \end{eqnarray}

\begin{remark}
\label{re-noIN}
We note that, to attain the previous expression, we have not used any disconjugacy hypotheses on operator $T_n[M]$. Moreover the proof is valid for $n-k$ even or odd. In other works, function $v$ solves problem $(P_v)$ for any linear operator defined in \eqref{e-Ln} and any $k\in \{1, \ldots,n-2\}$.
\end{remark}

We know, because of $g_{\bar M}(t,s)$ is of constant sign on $I \times I$ (see Lemma \ref{L::2}), that if $M=\bar{M}$ function $v$ must be of constant sign in $I$.

\begin{itemize}
	\item[Step 1.2.] \textit{If $v$ is of constant sign in $I$ then it can not have any zero in $(a,b)$.}
\end{itemize}
We are now going to see that while $v(t)$ is of constant sign in $I$ it can not have any zero in $(a,b)$. So the sign change comes on at $t=a$ or $t=b$. 

In order to do that, we are going to consider the decomposition of operator $T_n[M]$ made in Subsection \ref{SC::1}.
	
	Since $n-k$ is even, using Lemma \ref{L::2}, we know that operator $T_n[\bar{M}+\lambda] $ is, for $\lambda=0$, inverse positive on $X_k$. So, the characterization of $\lambda<0$ follows from Theorem \ref{T::6}.
		
	For $\lambda>0$,  $v\in C^n(I)$ is a solution of a linear differential equation, hence it is only allowed to have a finite number of zeros on $I$. Therefore, if  $v(t)\geq 0$, we have  that $v(t)>0$ for all $t\in I\backslash \left\lbrace t_0,\dots,t_\ell\right\rbrace $. In particular $v(t)>0$ for a.e. $t\in I$. Thus
	
	\begin{equation}\label{Ec::lll}
T_n[\bar{M}]\,v(t)=-\lambda\,v(t)<0\,\ \text{for a.e. } t\in I\,.\end{equation}

As we have shown in Subsection \ref{SC::1}, we know that
	\[T_n[\bar{M}]\,v(t)=v_1(t)\dots v_n(t)\,\dfrac{d}{dt}\left( \dfrac{1}{v_n(t)}\,T_{n-1}\,v(t)\right).\]

Since for every $k=1,\dots,n$, $v_k\in C^{n-k+1}(I)$ and $v_k(t)>0$ on $I$, we conclude that
 $\dfrac{1}{v_n(t)}\,T_{n-1}\,v(t)$ must be decreasing on $I$.
	
Therefore, since $v_n(t)>0$ on $I$ we have that $T_{n-1}v(t)$ can vanish at most once in $I$.
	
Arguing by recurrence, we have that $T_0\,v(t)=v(t)$ can have at most $n$ zeros on $I$ (multiple zeros being counted according to their multiplicity) while $v(t)\geq 0$. 

On the other hand, because of the boundary conditions \eqref{e-k-n-k}, we know  that $v$ vanishes $n-1$ times on $a$ and $b$, hence it can not have a double zero on $(a,b)$. This implies that sign change can not come from $(a,b)$.

\begin{itemize}
	\item[Step 1.3]\textit{Change sign of $v$ at $t=a$ and $t=b$.}
\end{itemize}
	
We are now going to see that the sign change cannot come from a neighborhood of $t=a$.

	Since $n-k$ is even, as we have proved before,  $v^{(k-1)}(a)>0$ for all $M\in \mathbb{R}$, which implies, since $v(a)=\dots =v^{(k-2)}(a)=0$, that  $v(t)={g^2_M}_{s^{n-k}}(t,a)$ is always positive on a neighborhood of $t=a$. This allows us to affirm that  Green's function, $g_M(t,s)$, is positive on a neighborhood of $(a,a)$.
	
	Using Step 1.2, we have that $v$ will keep constant sign on $I$  while  $v^{(n-k)}(b)=0$ is not satisfied, i.e.,  while an eigenvalue of  $T_n[\bar{M}]$ on $X_{k-1}$ is not attained.
	
Or equivalently, if $M\in[\bar{M},\bar{M}-\lambda_2']$ then $g_M(t,s)$ remains positive on a right neighborhood of $s=a$. Moreover, by Theorem \ref{T::6}, we deduce that  $g_M(t,s)$ oscillates in $I \times I$ for all $M>\bar{M}-\lambda_2'$.

	\begin{itemize}
\item[Step 1.4.] \textit{Study of function $u$.}

	\end{itemize}
	
	In order to analyse the behaviour of the Green's function on a left neighborhood of  $s=b$, we work now with the function $u$ defined in \eqref{d-u}.

	Using the same arguments than of $v$, we conclude that $u$ is the unique solution of the following problem, which we denote as $(P_u)$:
\begin{eqnarray}
\nonumber T_n[M]\,u(t)&=&0\,,\quad t\in I\,,\\\nonumber
u(a)&=&\dots = u^{(k-1)}(a)=0,\\\nonumber
u(b)&=&\dots = u^{(n-k-2)}(b)=0\,,\\\nonumber
u^{(n-k-1)}(b)&=&(-1)^{k-1}\,.
 \end{eqnarray}
	
As in Remark \ref{re-noIN}, we have that this property does not depend either on the disconjugacy of operator $T_n[M]$ nor if $n-k$ is even or odd.

	Using  analogous arguments to the ones done with $v$, we can prove that sign change cannot come on the open interval $(a,b)$
	
	Moreover, from condition $u^{(n-k-1)}(b)=(-1)^{k-1}$, sign change of $u$ cannot appear on $b$.
	
	So $u$ is of constant sign in $I$ until $u^{(k)}(a)=0$ is  verified, i.e., while an eigenvalue of $T_n[\bar{M}]$ on $X_{k+1}$ does not exist. Or, equivalently, while $M\in[\bar{M},\bar{M}-\lambda_2'']$.  
	
	Thus we have that if $M$ is on that interval, Green's function $g_M(t,s)$ has constant sign on a left neighborhood of $s=b$, but once $M>\bar{M}-\lambda_2''$ Green's function oscillates in $I \times I$.\\
		
		As a consequence of Step 1, we deduce that interval $(\bar{M}-\lambda_1,\bar{M}-\lambda_2]$  cannot be enlarged. Moreover we have also proved that the Green's function has constant sign on a neighborhood of $s=a$ and of $s=b$ for all $M$ in such interval.
		
\begin{itemize}
	\item[Step 2.] \textit{Behavior of Green's function on a neighborhood of $t=a$ and $t=b$.}
\end{itemize}	
	
	Now, let us see what happens on a neighborhood of $t=a$ and $t=b$. In order to do that, we are going to use the  operator $\widehat{T}_n[(-1)^n\,\bar{M}]$ defined in \eqref{Ec::Tg} and the relation between $g_M(t,s)$ and $\widehat{g}_{(-1)^n\,M}(t,s)$ given in (\ref{Ec::gg1}).
	
	Arguing as in Step 1, we will obtain the values of the real parameter $M$ for which $\widehat{g}_{(-1)^n\,M}(t,s)$ is of constant sign on a neighborhood of $s=a$ and  $s=b$. Once we have done it, we will be able to apply such property to the behaviour of $g_M(t,s)$ on a neighborhood of $t=a$ or $t=b$.
	
The analogous problem for operator $\widehat{T}_n[(-1)^n\,M]$ related to problem (\ref{e-Ln})-(\ref{e-k-n-k}) is given by
	\[\left\lbrace \begin{array}{l}
	\widehat{T}_n[(-1)^n\, M] v(t)=0, \quad t \in I\,,\\\\
	v(a)=\cdots=v^{(n-k-1)}(a)=0\,,\\\\
	v(b)=\cdots=v^{(k-1)}(b)=0\,.\end{array}\right. \]
	
	Theorem \ref{T::2} implies that equation $T_n^*[\bar{M}]\,u(t)=0$ is disconjugate on $I$. So, the same holds with $\widehat{T}_n[(-1)^n\,\bar{M}]\,u(t)=0$. Reasoning as in Step 1, we are able to prove that $\widehat{g}_{(-1)^n\,M}(t,s)$ has constant sign on a neighborhood of  $s=a$, while an eigenvalue of $\widehat{T}_n[(-1)^n\,\bar{M}]$ on $X_{n-k-1}$, let it be denoted as $\widehat{\lambda}_2''$, is not attained.
	
	This fact is equivalent to the existence of an eigenvalue of $T^*_n[\bar{M}]$ on $X_{n-k-1}$, that will be $(-1)^n \widehat{\lambda}_2''$. Now, using the fact that the real eigenvalues of an operator coincide with those of the adjoint operator, we conclude that $\lambda_2''=(-1)^n \widehat{\lambda}_2''$ is the biggest negative eigenvalue of   $T_n[\bar{M}]$ on $X_{n-(n-k-1)}=X_{k+1}$  and $\widehat{g}_{(-1)^n\,M}(t,s)$ is of constant sign on a right neighborhood of $s=a$ while $M\in[\bar{M},\bar{M}-\lambda_2'']$. So, Green's function of problem (\ref{e-Ln})-(\ref{e-k-n-k}), $g_M(t,s)$, does not oscillate on a right neighborhood of $t=a$.
	
	Analogously, arguing as before, we know that $\widehat{g}_{(-1)^n\,M}(t,s)$ does not oscillate on a left neighborhood of $s=b$ while an eigenvalue of $\widehat{T}_n[M]$ on $X_{n-k+1}$ is not attained, which is equivalent to the existence of an eigenvalue of $T_n[M]$ on $X_{k-1}$. If $M\in[\bar{M},\bar{M}-\lambda_2']$ we can affirm that Green's function of operator  $\widehat{T}_n[(-1)^n\,M]$, $\widehat{g}_{(-1)^n\,M}(t,s)$, will not oscillate on a left neighborhood of $s=b$, as consequence Green's function of problem (\ref{e-Ln})-(\ref{e-k-n-k}), $g_M(t,s)$, will not oscillate on a left neighborhood of $t=b$.\\
	
	As a consequence of the two previous Steps, we have already proved that if $M\in [\bar{M},\bar{M}-\lambda_2]$ then Green's function remains of constant sign on a neighborhood of the boundary of $I \times I$. And if $M>\bar{M}-\lambda_2$  Green's function oscillates on $I \times I$.
	
	\begin{itemize}
\item[Step 3.] \textit{The Green's function does not  become to change sign on $(a,b)\times (a,b)$.}
	\end{itemize}

	In this Step  we will prove that the oscillation of Green's function related to problem (\ref{e-Ln})-(\ref{e-k-n-k}) must begin on the boundary of $I \times I$. Using Theorem \ref{T::d1} we have that, provided it has non-negative sign on $I\times I$, $g_M$ decreases in $M$. 
	
	As consequence, once we prove that $g_M$ cannot have a double zero on $(a,b)\times (a,b)$, the change of sign must start on the boundary of $I\times I$.\\
	
	Let us see that if $g_M(t,s)\geq 0$ in $I\times I$ then $g_M(t,s)>0$ in $(a,b)\times (a,b)$.
	
	Denote, for a fixed $s\in (a,b)$, $w_s(t)=g_M(t,s)$. By definition, denoting, as in Step 1, $\lambda=M -\bar{M}$, we have that 
		\[T_n[\bar{M}]\,w_s(t)+\lambda\,w_s(t)=0\,,\quad t\in I\,,\ t\neq s\,.\]
	
	Since $g_{\bar M}\ge 0$ on $I \times I$, the behaviour for $M<\bar{M}$ has been characterized in Lemma \ref{L::2} and Theorem \ref{T::6}. 
	
	So we must pay our attention on the situation $M > \bar{M}$, i.e. $\lambda>0$. In such a case, since, as in Step 1.2, we have that  $w_s(t)\geq 0$ has a finite number of zeros in  $I$, we know that
	\[T_n[\bar{M}]\,w_s(t)=-\lambda \,w_s(t)<0 \text{ for a. e. } t\in I\,.\]
	
	Using \eqref{e-descomp} and \eqref{e-Tk}, we have that
	 $$T_n[\bar{M}]\,w_s(t)=v_1(t)\dots v_n(t) \,T_n \,w_s(t),$$ 
	 with $v_k>0$ on $I$ for  $k=1,\dots,n$. In particular, it is satisfied that $T_n w_s(t)<0$ a.e. in $I$.
	
	Notice that, for all $s \in (a,b)$, it is satisfied that $w_s \in C^{n-2}(I)$ and $w_s^{(n-1)}(s^+)-w_s^{(n-1)}(s^-)=1$.
	Therefore, due to the definition of $T_n[\bar{M}]$ and expression \eqref{Ec::7}, we have that $\dfrac{1}{v_n(t)}T_{n-1} w_s(t)$ is a continuous function on $[a,s) \cup (s,b]$.
	
	Since $T_n w_s(t)=\dfrac{d}{dt}\left( \dfrac{1}{v_n(t)}T_{n-1} w_s(t)\right) <0$ for $t \neq s$, we can affirm that $\dfrac{1}{v_n(t)}\,T_{n-1} w_s(t)$ is a decreasing function on $I$ with a positive jump at $t=s$. So, it can have, at most, two zeros  in $I$, (see Figure \ref{Fig::1}).
	
	\begin{figure}[h]
\centering
\includegraphics[width=0.3\textwidth]{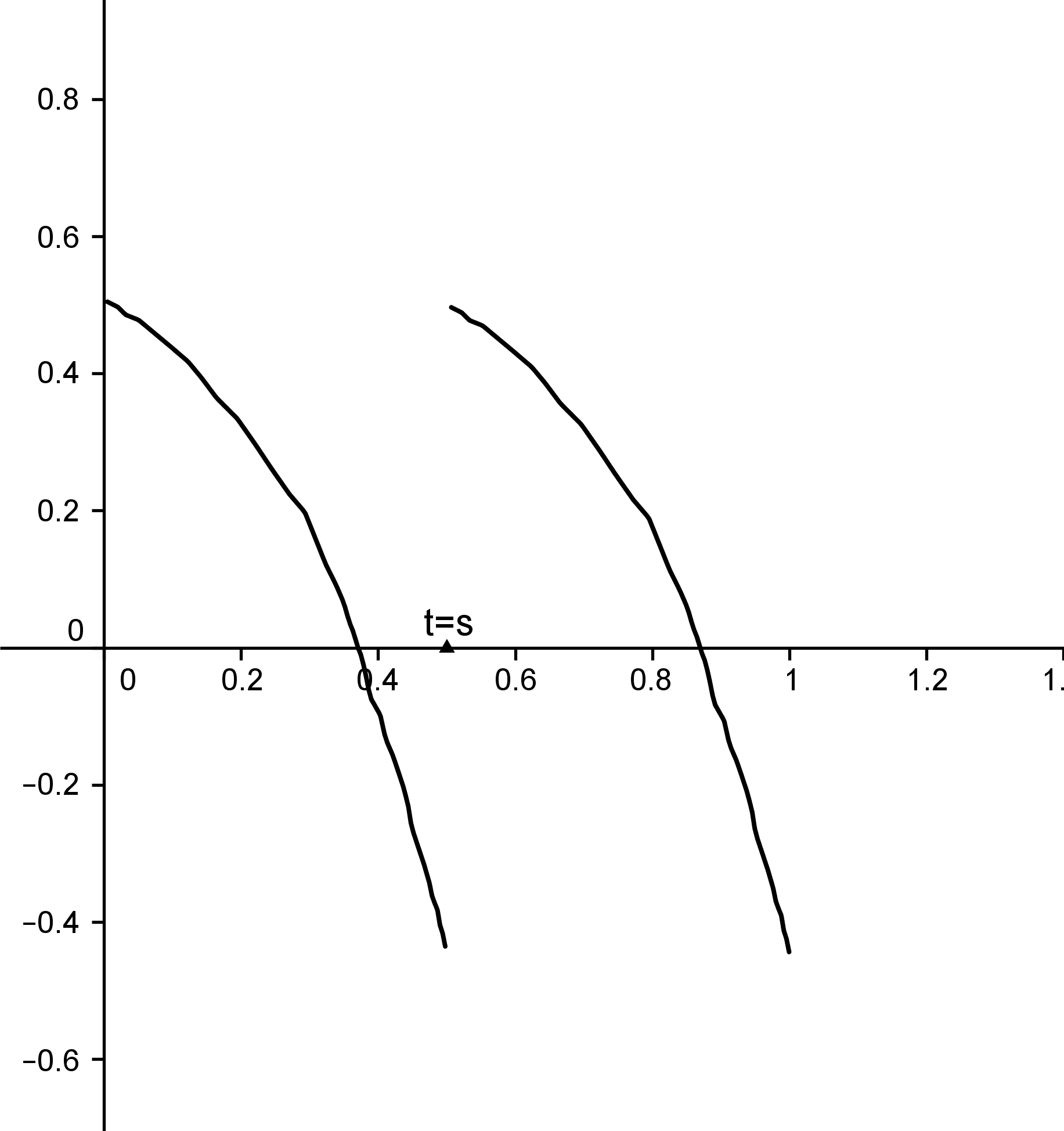}
\caption{\scriptsize{$\dfrac{1}{v_n(t)}\,T_{n-1} w_s(t)$, maximal oscillation with $I=[0,1]$}\label{Fig::1}}

	\end{figure}

	Even we can not guarantee that $T_{n-1} w_s(t)$ is decreasing, since $v_n>0$ on $I$, we  conclude that it has the same sign as $\dfrac{1}{v_n(t)}\,T_{n-1} w_s(t)$, i.e, it can have at most two zeros on $I$.
	
	By the other hand, using equation \eqref{Ec::7} again, we conclude that $\dfrac{1}{v_{n-1}(t)}\,T_{n-2} w_s(t)$ is a continuous function on $I$. Now, \eqref{e-Tk} tell us that  $\dfrac{1}{v_{n-1}(t)}\,T_{n-2} w_s(t)$ can reach at most $4$ zeros on $I$ (see Figure \ref{Fig::2}).

	\begin{figure}[h]
\centering
\includegraphics[width=0.3\textwidth]{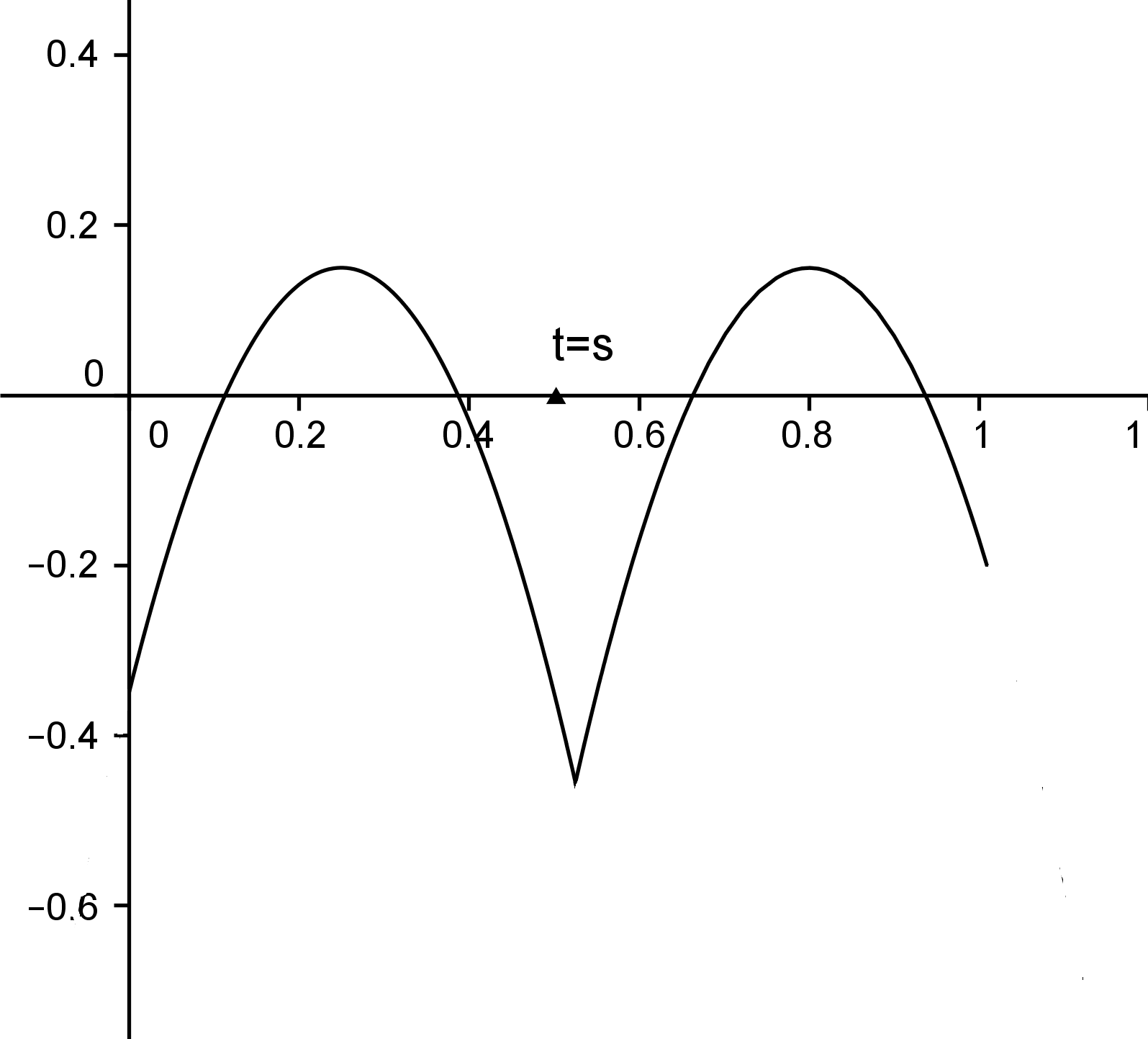}
\caption{\scriptsize{$\dfrac{1}{v_{n-1}(t)}\,T_{n-2} w_s(t)$, maximal oscillation with $I=[0,1]$}\label{Fig::2}}

	\end{figure}
	
	As before, we do not know intervals where $T_{n-2} w_s(t)$ is increasing or decreasing, but since $v_{n-1}(t)>0$ we conclude that it has the same sign as $\dfrac{1}{v_{n-1}(t)}\,T_{n-2} w_s(t)$, so it can reach at most $4$ zeros.
	
	Following this argument, since  $v_k>0$ on $I$ for $k=1,\dots,n$, we know that  $T_{n-2-h}w_s(t)$ can have not more than $4+h$ zeros on $I$ (multiple zeros being counted according to their multiplicity). In particular, $w_s(t)= T_0 w_s(t)$ can have  $n+2$ zeros at most, having $n$ in the boundary.

	This fact  allows $w_s$ to have a double zero on $(a,b)$. So, to show that such double root cannot exist, we need to prove that maximal oscillation is not possible. To this end, we point out that if for any $h$  it is verified that the sign of $T_{n-2-h}w_s(a)$ is equal to the sign of $T_{n-2-{h+1}}w_s(a)$ we lose a possible oscillation.
	
	Therefore, for maximal oscillation it must be satisfied
	
	\[\left\lbrace\begin{array}{ll} T_{n-h}w_s(a)>0, \quad & \text{if } h\text{ odd},\\&\\
	T_{n-h}w_s(a)<0, \quad & \text{if } h\text{ even.}\end{array}\right. \]
	
	However, since $w_s(t)\geq 0$ on $I$ and $w_s(a)=w_s'(a)= \dots =w_s^{(k-1)}(a)=0$, we deduce that $w_s^{(k)}(a)\geq 0$.
	
We can assume that $w_s^{(k)}(a)>0$ because, on the contrary, if $w_s^{(k)}(a)=0$ we would have $n+2$ zeros at most, having $n+1$ in the boundary. So, only a simple zero is allowed in the interior, which is not possible without oscillation.
	
	Therefore $ w_s^{(k)}(a)=w_s^{(n-(n-k))}(a)>0$. Since $n-k$ is even, using now \eqref{e-Tk-a}, we also know that $T_k w_s(a)>0$, which inhibits maximal oscillation.  
	
	So we conclude that if $g_M(t,s)\geq 0$ on $I\times I$ then $g_M(t,s)>0$ on $(a,b)\times(a,b)$, as we wanted to prove.\\
	
	As a consequence of the three previous Steps, we have described the set of the real parameters $M$ for which the Green's function is non-negative on $I \times I$ when $n-k$ is even. \\

	If $n-k$ is odd we can do similar arguments to achieve the proof. In the sequel, we enumerate the main ideas to be developed 
	
	\begin{itemize}
\item[Step 1.]{\color{white} 1}

\begin{itemize}
\item[Step 1.1.] It has no modifications. 
\item[Step 1.2.] In equality (\ref{Ec::lll}) we have $\lambda <0$ and $v(t)<0$ a.e. in $I$, so it remains true and we can proceed analogously.
\item[Step 1.3.] In this case, we have that $v^{(k-1)}(a)<0$. 
Our attainment in this Step is that $g_M(t,s)$ remains negative while $M\in [\bar{M}-\lambda_2',\bar{M}]$ in a neighborhood of $s=a$ and oscillates for all $M<\bar{M}-\lambda_2'$.
\item[Step 1.4.] The arguments are not modified, but the final achievement is that $g_M(t,s)$ is negative in a neighborhood of $s=b$ for $M\in[\bar{M}-\lambda_2'',\bar{M}]$ an oscillates for all $M<\bar{M}-\lambda_2''$.
	\end{itemize}
	\item[Step 2.] Using the same arguments we conclude that the interval where $g_M(t,s)$ is non-positive on the boundary of $I\times I$ is $[\bar{M}-\lambda_2,\bar{M}]$.
	\item[Step 3.] In this case  we have that $w^{(k)}_s(a)=w^{(n-(n-k))}_s(a)<0$, with $n-k$ odd contradicting maximal oscillation too.
	\end{itemize}

Thus, our result is proved. \qed\\

As a direct consequence of the arguments used in Step 1.3, without assuming the existence of $\bar{M}\in\mathbb{R}$ for which equation $T_n[\bar{M}]\,u(t)=0$  is disconjugate on $I$, we arrive at the following result.

\begin{corollary}
\label{cor:noIN}
Let $T_n[M]$ be defined as in \eqref{e-Ln}. Then the two following properties hold:

If $n-k$ is even, then it does not exist $M\in\mathbb{R}$ such that operator $T_n[M]$ is inverse negative in $X_k$.

If $n-k$ is odd, then it  does not exist $M\in\mathbb{R}$ such that operator $T_n[M]$ is inverse positive in $X_k$.
\end{corollary}

\begin{proof}
It is enough to take into account that $v$, defined in \eqref{d-v}, is the unique solution of problem $(P_v)$.  Since $v^{(k-1)}(a)=(-1)^{n-k}$ we conclude that, if $n-k$ is even, Green's function has positive values in any neighborhood of $(a,a)$ and negative when $n-k$ is odd. 

So, the result holds from Theorem \ref{T::in1}.
\end{proof}

\section{Particular cases}

In order to obtain the eigenvalues of particular problems we calculate a fundamental system of solutions $y_1[M](t),\dots,y_n[M](t)$ of equation \eqref{e-Ln} where every $y_k[M](t)$ verifies the initial conditions

\[y_k^{(n-k)}[M](a)=1\,,\quad y_k^{(n-j)}[M](a)=0\,,\ j=1,\dots,n\,,\ j\neq k\,.\]

Then we denote  the $n-1$ Wronskians as
\[W^n_k[M](t)=\left| \begin{array}{ccc}
y_1[M](t)&\dots &y_k[M](t)\\
y_1'[M](t)&\dots&y_k' [M](t)\\
&\vdots&\\
y_1^{(k-1)}[M](t)&\dots&y_k^{(k-1)}[M](t)\end{array}\right| \,,\quad k=1,\dots,n-1\,.\]

As a consequence of the characterization done in \cite[Chapter 3, Lemma 12]{Cop}, we deduce that the eigenvalues of problem \eqref{e-Ln} in $X_k$ are given as the $\lambda\in \mathbb{R}$ for which $W_{n-k}[-\lambda](b)=0$. So, in the sequel, we will use this method to find the eigenvalues of the different considered problems.

\subsection{Operator $T_n[M]\,u(t)\equiv u^{(n)}(t)+M\,u(t)$}

First of all, we are going to consider problems where $T_n[M]\,u(t)\equiv u^{(n)}(t)+M\,u(t)$, with  $[a,b]=[0,1]$. 

In this kind of problems, for $M=0$, $u^{(n)}(t)=0$ is always disconjugate, see \cite[Chapter 3]{Cop}. So, hypotheses of Theorem \ref{T::5} are satisfied.

\begin{remark}
	Note that adjoint equation to problem $T_n[M]\,u=0,  u \in X_k$  is given by 
	$$T_n^*[M]\,u(t)=(-1)^n\,u^{(n)}(t)+M\,u(t)=0, \quad u \in X_{n-k}.$$
	
	So, if we have that $\lambda_i$ is an eigenvalue of $u^{(n)}$ in $X_k$, it is also an eigenvalue of $(-1)^n\,u^{(n)}$ in $X_{n-k}$. Thus, $(-1)^n\,\lambda_i$ is an eigenvalue of $u^{(n)}$ in $X_{n-k}$.
	
	As consequence, we only need to obtain first $\left\lfloor\dfrac{n}{2}\right\rfloor$ Wronskians, where $\lfloor\cdot \rfloor$ means the floor function.
\end{remark}

\begin{itemize}
	\item[-] Order 2
\end{itemize}

The eigenvalues of operator $u''(t)$ in $X_1$ must satisfy $W_1^2[\lambda](1)=0$, which can be replaced by the following equation
\begin{equation}
\sin (\sqrt{-\lambda})=0\,,
\end{equation}
so it closest to zero negative eigenvalue is $\left( \lambda_2^1\right) ^2=-\pi^2$.

And so, we can affirm that Green's function related to operator $u''(t)+M\,u(t)$ is negative if, and only if, $M\in\left( -\infty\,,\,\pi ^2\right) $.

This result has been already obtained in different references (See \cite{Cab} and references therein), but here it is not necessary to have the expression of the Green's function.
\begin{itemize}
	\item[-] Order 3
\end{itemize}

 $ \lambda_3^1\approxeq 4.23321$ is the least positive solution of $W_1^2[\lambda^3]=0$, which is equivalent to the equation 
\begin{equation*}
\cos \left(\frac{1}{2} \sqrt{3} \lambda \right)-\sqrt{3} \sin \left(\frac{1}{2} \sqrt{3} \lambda  \right)=e^{\frac{-3\,\lambda}{2}}\,.
\end{equation*}

Then, the least positive eigenvalue of operator $u^{(3)}(t)$ in $X_1$  is $\left( \lambda_3^1\right)  ^3$ and the biggest negative eigenvalue of operator $u^{(3)}(t)$ in $X_2$  is $-\left(  \lambda_3^1\right)  ^3$.

So, we can affirm that Green's function of operator $u^{(3)}(t)+M\,u(t)$
\begin{itemize}
	\item in $X_1$ is positive if, and only if, $M\in \left( -\left(  \lambda_3^1\right)  ^3,\left(  \lambda_3^1\right)  ^3\right] $,
	\item in $X_2$ is positive if, and only if, $M\in \left[ -\left( \lambda_3^1\right)  ^3,\left( \lambda_3^1\right) ^3\right) $.
\end{itemize}

This result has been obtained by means of the explicit form of Green's function in \cite{ML}.

\begin{itemize}
	\item[-] Order 4
\end{itemize}

$\lambda_4^1\approxeq 5.553$ is the least positive solution of $W_1^4[\lambda^4](1)=0$, simplifying that expression we have
\begin{equation*}
\tan\left( \dfrac{\lambda}{\sqrt{2}}\right) =\tanh\left( \dfrac{\lambda}{\sqrt{2}}\right) \,.
\end{equation*}

 $\lambda_4^2\approxeq4.73004$ is the least positive solution of $W_2^4[-\lambda^4](1)=0$, which can be expressed as
\begin{equation*}
\cos (\lambda ) \cosh (\lambda )=1\,.
\end{equation*}

The biggest negative eigenvalue of operator $u^{(4)}(t)$ in $X_1$ and $X_3$ is given by $-\left( \lambda_4^1\right) ^4$. 

The least positive eigenvalue of operator $u^{(4)}(t)$ in $X_2$ is $\left( \lambda_4^2\right) ^4$.\\

Therefore, we can affirm without calculating it explicitly, that Green's function related to the operator $u^{(4)}(t)+M\,u(t)$
\begin{itemize}
	\item in $X_1$ and $X_3$ is negative if, and only if, $M\in\left[ -\left( \lambda_4^2\right) ^4,\left( \lambda_4^1\right) ^4\right)$.

	\item in $X_2$ is positive if, and only if, $M\in\left( -\left( \lambda_4^2\right) ^4,\left( \lambda_4^1\right) ^4\right]$.
	
\end{itemize}

These results have been obtained using the explicit form of Green's function in \cite{CE} and \cite{CaFe}.

\begin{itemize}
	\item[-] Order 5
\end{itemize}

We can obtain $\lambda_5^1\approxeq  6.94867$ and $\lambda_5^2\approxeq 5.64117$  as the least positive solution of $W_1^5[\lambda^5](1)=0$ and $W_2^5[-\lambda^5](1)=0$, respectively. But the equations obtained are too complicate to show here and they have not so much interest.

The least positive eigenvalue of operator $u^{(5)}(t)$ in $X_1$ is $\left( \lambda_5^1\right) ^5$.

The biggest negative eigenvalue of operator $u^{(5)}(t)$ in $X_2$ is $-\left( \lambda_5^2\right) ^5$.

The least positive eigenvalue of operator $u^{(5)}(t)$ in $X_3$ is $\left( \lambda_5^2\right) ^5$.

The biggest negative eigenvalue of operator $u^{(5)}(t)$ in $X_4$ is $-\left( \lambda_5^1\right) ^5$.\\

Therefore, we conclude without calculating it explicitly, that Green's function related to the operator $u^{(5)}(t)+M\,u(t)$
\begin{itemize}
	\item in $X_1$ is positive if, and only if, $M\in\left( -\left( \lambda_5^1\right) ^5,\left( \lambda_5^2\right) ^5\right]$.
	
	\item in $X_2$ is negative if, and only if, $M\in\left[ -\left( \lambda_5^2\right) ^5,\left( \lambda_5^2\right) ^5\right)$.
	
	\item in $X_3$ is positive if, and only if, $M\in\left( -\left( \lambda_5^2\right) ^5,\left( \lambda_5^2\right) ^5\right]$.
	
	\item in $X_4$ is negative if, and only if, $M\in\left[ -\left( \lambda_5^2\right) ^5,\left( \lambda_5^1\right) ^5\right)$.
\end{itemize}

\begin{itemize}
	\item[-] Order 6
\end{itemize}

	$\lambda_6^1\approxeq 8.3788$ is the least positive solution of $W_1^6[\lambda^6](1)=0$, which is equivalent to
	
	\[\sin (\lambda )-\sqrt{3} \cos \left(\frac{\lambda }{2}\right) \sinh \left(\frac{\sqrt{3} \lambda }{2}\right)+\sin \left(\frac{\lambda }{2}\right) \cosh
	\left(\frac{\sqrt{3} \lambda }{2}\right)=0\,.\]
	
$\lambda_6^2\approxeq 6.70763$ is the least positive solution of  $W_2^6[-\lambda^6](1)=0$, which we can express as
	
{\footnotesize 	\[-3 e^{\lambda /2} \left(e^{2 \lambda }+1\right)+\sqrt{3} \left(e^{\lambda }-1\right)^3 \sin \left(\frac{\sqrt{3} \lambda }{2}\right)+\left(e^{\lambda }+1\right)^3 \cos
	\left(\frac{\sqrt{3} \lambda }{2}\right)-2 e^{3 \lambda /2} \cos \left(\sqrt{3} \lambda \right)=0\,.\]}
	
	 $\lambda_6^3\approxeq 6.28319^6$ is the least positive solution of $W_3^6[\lambda^6](1)=0$, which can be represented as the first positive root of the following equation
	\[\sin (\lambda ) \left(-\cos (\lambda )+\cosh \left(\sqrt{3} \lambda \right)+4\right)-8 \sin \left(\frac{\lambda }{2}\right) \cosh \left(\frac{\sqrt{3} \lambda }{2}\right)=0\,.\]

	The biggest negative eigenvalue of operator $u^{(6)}(t)$ in $X_1$ and $X_5$ is given by $-\left(\lambda_6^1\right)^6$.

	The least positive eigenvalue of operator $u^{(6)}(t)$ in $X_2$ and $X_4$ is $\left(\lambda_6^2\right)^6$.

	The biggest negative eigenvalue of operator $u^{(6)}(t)$ in $X_3$ is $-\left(\lambda_6^3\right)^6$.\\

	Hence, we can affirm without calculating it explicitly, that Green's function related to  operator $u^{(6)}(t)+M\,u(t)$
	\begin{itemize}
\item in $X_1$ or in $X_5$ is negative if, and only if, $M\in\left[ -\left( \lambda_6^2\right) ^6,\left( \lambda_6^1\right) ^6\right)$.

\item in $X_2$ or in $X_4$ is positive if, and only if, $M\in\left( -\left( \lambda_6^2\right) ^6,\left( \lambda_6^3\right) ^6\right]$.

\item in $X_3$ is negative if, and only if, $M\in\left[ -\left( \lambda_6^2\right) ^6,\left( \lambda_6^3\right) ^6\right)$.
	
	\end{itemize}
	
	\begin{itemize}
\item[-] Order 7
	\end{itemize}
	
	We are not able to obtain analytically the eigenvalues of operator $u^{(7)}(t)$, but we can obtain them numerically.
	
	The least positive eigenvalue of this operator in $X_1$ is $\left( \lambda_7^1\right) ^7$, where $\lambda_7^1\approxeq 9.82677$.
	
	The biggest negative eigenvalue in $X_2$ is $-\left( \lambda_7^2\right) ^7$, where $\lambda_7^2\approxeq 7.85833$.
	
	The least positive eigenvalue in $X_3$ is $\left( \lambda_7^3\right) ^7$, where $\lambda_7^3\approxeq 7.1347$.
	
	The biggest negative eigenvalue in $X_4$ is $-\left( \lambda_7^3\right) ^7$.
	
	The least positive eigenvalue in $X_5$ is $\left( \lambda_7^2\right) ^7$.
	
	The biggest negative eigenvalue in $X_6$ is 
	$-\left( \lambda_7^1\right) ^7$.\\
	
	So, we conclude without calculating it explicitly, that Green's function related to the operator $u^{(7)}(t)+M\,u(t)$
	\begin{itemize}
\item in $X_1$ is positive if, and only if, $M\in\left( -\left( \lambda_7^1\right) ^7,\left( \lambda_7^2\right) ^7\right]$.

\item in $X_2$ is negative if, and only if, $M\in\left[ -\left( \lambda_7^3\right) ^7,\left( \lambda_7^2\right) ^7\right)$.

\item in $X_3$ is positive if, and only if, $M\in\left( -\left( \lambda_7^3\right) ^7,\left( \lambda_7^3\right) ^7\right]$.

\item in $X_4$ is negative if, and only if, $M\in\left[ -\left( \lambda_7^3\right) ^7,\left( \lambda_7^3\right) ^7\right)$.

\item in $X_5$ is positive if, and only if, $M\in\left( -\left( \lambda_7^2\right) ^7,\left( \lambda_7^3\right) ^7\right]$.
	
\item in $X_6$ is negative if, and only if, $M\in\left[ -\left( \lambda_7^2\right) ^7,\left( \lambda_7^1\right) ^7\right)$.

	\end{itemize}

	\begin{itemize}
\item[-] Order 8
	\end{itemize}
	$\lambda_8^1\approxeq 11.2846$, $\lambda_8^2\approxeq 9.06306$, $\lambda_8^3\approxeq 8.09971$ and $\lambda_8^4\approxeq7.81871$ can be obtained analytically as  the least positive solution of $W_1^8[\lambda^8](1)=0$, $W_2^8[-\lambda^8](1)=0$, $W_3^8[\lambda^8](1)=0$ and $W_4^8[-\lambda^8](1)=0$ respectively, but their expressions are too big to show it here and they do not bring any important information.
	
	The biggest negative eigenvalue of operator $u^{(8)}(t)$ in $X_1$ and $X_7$ is given by $-(\lambda_8^1)^8$.

The least positive eigenvalue of operator $u^{(8)}(t)$ in $X_2$ and $X_6$ is given by $(\lambda_8^2)^8$.

The biggest negative eigenvalue of operator $u^{(8)}(t)$ in $X_3$ and $X_5$ is given by $-(\lambda_8^3)^8$.

The least positive eigenvalue of operator $u^{(8)}(t)$ in $X_4$ is $(\lambda_8^4)^8$.\\

So, we can affirm without calculating it explicitly, that Green's function related to the operator $u^{(8)}(t)+M\,u(t)$

\begin{itemize}
	\item in $X_1$ or in $X_7$ is negative if, and only if, $M\in\left[ -\left( \lambda_8^2\right) ^8,\left( \lambda_8^1\right) ^8\right)$.

	\item in $X_2$ or in $X_6$ is positive if, and only if, $M\in\left( -\left( \lambda_8^2\right) ^8,\left( \lambda_8^3\right) ^6\right]$.

	\item in $X_3$ or in $X_5$ is negative if, and only if, $M\in\left[ -\left( \lambda_8^4\right) ^8,\left( \lambda_8^3\right) ^8\right)$.

	\item in $X_4$ is positive if, and only if, $M\in\left( -\left( \lambda_8^4\right) ^8,\left( \lambda_8^3\right) ^8\right]$.

	\end{itemize}
	
As we have said before, third-order problems were explicitly calculated on \cite{ML}. And fourth-order problems were calculated on \cite{CE} in $X_2$ and on \cite{CaFe} in $X_1$ and $X_3$, respectively.
	But, in all of these cases were necessary to obtain the expression of Green's function and analyse it.
	
	Moreover, in all the problems treated on \cite{CE, CaFe, ML} it is also satisfied that the open optimal interval where Green's function is of constant sign coincide with the optimal interval where the equation \eqref{e-Ln} is disconjugate.
	
	However, in \cite[Theorem 2.1]{CaSaa} it is proved the following characterization of the interval of disconjugacy:
\begin{theorem}\label{T::aut}
		Let $\bar{M}\in\mathbb{R}$ and $n\geq 2$ be such that $T_n[\bar{M}]\,u(t)=0$ is a disconjugate equation on $I$. Then, $T_n[M]\,u(t)=0$ is a disconjugate equation on $I$ if, and only if, $M\in(\bar{M}-\lambda_1,\bar{M}-\lambda_2)$, where
		\begin{itemize}
			\item  $\lambda_1=+\infty$ if $n=2$ and, for $n>2$, $\lambda_1>0$ is the minimum of the least positive eigenvalues on $T_n[\bar{M}]$ in $X_k$, with $n-k$  even.
			\item $\lambda_2<0$ is the maximum of the biggest negative eigenvalues on $T_n[\bar{M}]$ in $X_k$, with $n-k$  odd.
		\end{itemize}
	\end{theorem}

	As consequence we have 	that  the  interval of constant sign of the Green's function and the one of the disconjugacy for the linear operator are not the same in general. We have already proved (see Lemma \ref{L::2}) that while equation \eqref{e-Ln} is disconjugate its related Green's function must be of constant sign. So, if both intervals do not coincide, the optimal interval where the equation \eqref{e-Ln} is disconjugate must be contained in the open optimal interval where Green's function is of constant sign .
	
	If, using the characterization given in Theorem \ref{T::aut}, we calculate the optimal interval on $M$ of disconjugacy for equation 
	\begin{equation*}
	u^{(5)}(t)+M\,u(t)=0\,,\quad  t\in[0,1].
	\end{equation*}
	
	We have that it is given by $(-\left( \lambda_5^2\right) ^5,\left( \lambda_5^2\right) ^5)$.
	
	But, as we have shown before, Green's function related to the problem on the space $X_1$ remains positive on the interval $(-\left( \lambda_5^1\right) ^5,-\left( \lambda_5^2\right) ^5]$. So, its biggest open interval is strictly bigger than the optimal interval of disconjugacy.

	\begin{remark}
In this kind of problems,  if $\lambda$ is an eigenvalue on $[0,1]$, then $\dfrac{\lambda}{(b-a)^n}$ is an eigenvalue on $[a,b]$.

So, we can obtain our conclusions about Green's function' sign on any arbitrary interval $[a,b]$.
	\end{remark}
	
	\subsection{Operators with constant coefficients}
	This characterization of the interval where Green's function is of constant sign is also useful for those problems which have more non-nulls coefficients .

	For example we can consider the operator of fourth order
	\begin{equation}\label{Ec::10}
	T_n[M]\,u(t)\equiv u^{(4)}(t)+10\,u^{(3)}(t)+10\,u''(t)+10\,u'(t)+M\,u(t), \quad t\in[0,1].
	\end{equation}
	
	We can show, using the characterization given in Theorem \ref{T::4}, that $T_n[0]\,u(t)=0$ is a disconjugate equation on $[0,1]$ and, so, Theorem \ref{T::5} holds.
	
	First, we calculate numerically the closest to zero eigenvalues in each  $X_k$, $k=1,2,3$. 
	
	\begin{itemize}
\item The biggest negative eigenvalue in $X_1$ is $-(7.02782)^4$.
\item The least positive eigenvalue in $X_2$ is $(5.27208)^4$.
\item The biggest negative eigenvalue in $X_3$ is $-(5.97041)^4$.
	\end{itemize}
	
	Realize that in this case we need to obtain the three correspondents Wronskians because it is not possible to connect eigenvalues in $X_1$ with those in $X_3$ by means of its corresponding adjoint equation.
	
	So, we  conclude that Green's function related to operator $T_n[M]\,u(t)$ defined in \eqref{Ec::10} satisfies that
	\begin{itemize}
\item in $X_1$ is  negative if, and only if, $M\in[-(5.27208)^4,(7.02782)^4)$,
\item in $X_2$ is positive if, and only if, $M\in(-(5.27208)^4,(5.97041)^4]$,
\item in $X_3$ is negative if, and only if, $M\in[-(5.27208)^4,(5,97041)^4)$.
	\end{itemize}
	
	Notice that in this case the interval of disconjugation is $(-(5.27208)^4,(5.97041)^4)$. So, we have obtained an example of a fourth order  equation in which its interval of disconjugation does not coincide with the biggest open interval where Green's function is of constant sign in $X_1$.
	
	In the sequel, we show an example where operator $T_n[M]$ does not verify disconjugation hypothesis for $\bar{M}=0$.
	
	If we choose the operator
	\begin{equation}\label{Ec::11}
	T_n[M]\,u(t)\equiv u^{(4)}(t)+10\,u^{(3)}(t)+550\,u'(t)+M\,u(t)\,,\quad t\in[0,1]\,.
	\end{equation}
	
	We obtain that  equation $T_n[0]\,u(t)=0$ is not disconjugate on $[0,1]$, but if we analyse the equation $T_n[-600]\,u(t)=0$ we can affirm, by means of Theorem \ref{T::4}, that it is disconjugate on $[0,1]$.
	
	Hence, Theorem \ref{T::5} can be applied to the operator $T_n[-600]\,u(t)$.
	
	If we calculate the closest to zero eigenvalues we have
	\begin{itemize}
\item The biggest negative eigenvalue of $T_n[-600]\,u(t)$ in $X_1$ is $-9565.99$.
\item The least positive eigenvalue in $X_2$ is $11.5685$.
\item The biggest negative eigenvalue in $X_2$ is $-28.9753$.
	\end{itemize}
	
	Hence, using Theorem \ref{T::in1}, we can affirm that operator $T_n[M]\,u(t)$, defined in \eqref{Ec::11}, satisfies that
	\begin{itemize}
\item in $X_1$ is inverse negative if, and only if, $M\in [-600-11.5685,-600+9565.99)=[-611.5685,8965.99)$.
\item in $X_2$ is inverse positive if, and only if, $M\in (-600-11.5685,-600+28.9753]=(-611.5685,-571.0247]$.
\item in  $X_3$ is inverse negative if, and only if, $M\in[-600-11.5685,-600+28.9753)=[-611.5685,-571.0247)$.
\end{itemize}
	\subsection{Operators with non-constant coefficients}
	We have already seen that applying Theorem \ref{T::5} is much easier to calculate optimal intervals for $M$ where Green's function related to operator $T_n[M]\,u(t)$ than obtain Green's function expression explicitly. But, if we are referring to an operator with non-constant coefficients this characterization is even more useful because in the majority of the situations we are not able to  obtain the explicit  expression for the Green's function.
	
	Consider now the third order operator
	\begin{equation}\label{Ec::9}
	T_n[M]\,u(t)\equiv u^{(3)}(t)+t\,u'(t)+M\,u(t)\,,\quad t\in[0,1],
	\end{equation}
 for which, by means of Theorem \ref{T::4}, we can verify that equation $T_n[0]\,u(t)=0$ is disconjugate  on $[0,1]$. 
	
	If we calculate numerically the closest to zero eigenvalues of operator defined in \eqref{Ec::9} we obtain
	
	\begin{itemize}
\item $(4.19369)^3$ is the least positive eigenvalue of operator $T_n[0]\,u(t)$ in $X_1$.
\item $-(4.21255)^3$ is the biggest negative eigenvalue of operator $T_n[0]\,u(t)$ in $X_2$.
	\end{itemize}
	
	So, we can affirm
	
	 \begin{itemize}
	\item Green's function related to operator $T_n[M]\,u(t)$ in $X_1$ is positive if, and only if, $M\in (-(4.19369)^3,(4.21255)^3]$,
	\item Green's function related to operator $T_n[M]\,u(t)$ in $X_2$ is negative if, and only if, $M\in [-(4.19369)^3,(4.21255)^3)$.
	\end{itemize}
	
	We can also apply it to a fourth order operator whose eigenvalues were also obtained numerically. 
	\begin{equation}
	T_n[M]\equiv u^{(4)}(t)+e^{2\,t} u'(t)+M\,u(t)\,,\quad t\in[0,1].
	\end{equation}
	
	We can verify, by means of Theorem \ref{T::4} again, that $T_n[0]\,u(t)=0$ is disconjugate on $[0,1]$.
	
	If we calculate its eigenvalues we obtain
	\begin{itemize}
\item The biggest negative eigenvalue in $X_1$ is $-(5.5325)^4$.
\item The least positive eigenvalue in $X_2$ is $(4.7235)^4$.
\item The biggest negative eigenvalue in $X_3$ is $-(5.5815)^4$.
	\end{itemize}
	
	So, applying Theorem \ref{T::5}, we conclude that 
	\begin{itemize}
\item Green's function related to operator $T_n[M]\,u(t)$ in $X_1$ is negative if, and only if, $M\in [-(4.7235)^4,(5.5325)^4)$,
\item Green's function related to operator $T_n[M]\,u(t)$ in $X_2$ is positive if, and only if, $M\in (-(4.7235)^4,(5.5325)^4]$,
\item Green's function related to operator $T_n[M]\,u(t)$ in $X_3$ is negative if, and only if,  $M\in [-(4.7235)^4,(5.5815)^4)$.
	\end{itemize}

	\section{Disconjugacy hypothesis cannot be removed on Theorem \ref{T::5}}\label{S:w3}
	
	In this last section we show that the disconjugacy hypothesis on Theorem \ref{T::5} for some $M=\bar{M}$ cannot be avoided in general.
	
		To this end, we consider the operator \begin{equation}\label{Ec::op}
		T_4[M]\,u(t)=u^{(4)}(t)-1000\,u'(t)+M\,u(t)\,,\quad t\in[0,1]\,,
		\end{equation}
		 coupled with two-point boundary value conditions 
		 \begin{equation}\label{Ec::bc}
		 u(0)=u'(0)=u''(0)=u(1)=0.
		 \end{equation}
		 
	The equation \eqref{Ec::op} is not disconjugate for $M=0$, indeed:
	\[u(t)=\frac{-e^{10 (t-1)}-2 e^{5-5 t} \cos \left(5 \sqrt{3} (t-1)\right)+3}{3000}\,,\]
	is a solution of $T_4[0]\,u(t)=0$ with $5$ zeros on $[0,1]$.
	
In a first moment we will verify that  Green's function related to problem \eqref{Ec::op}-\eqref{Ec::bc} satisfies condition $(N_g)$ for $\bar{M}=0$. So, by means of Theorem \ref{T::7}, we know that $N_T=[-\mu,-\lambda_1)$ for some $\mu \ge 0$. 

In a second part, we will prove that $\mu\neq \lambda_2$, with $\lambda_2$ the first eigenvalue related to operator $T_4[0]$ on the space $X_2$.

As a consequence, we deduce that the validity of Theorem \ref{T::5} is not ensured when the disconjugacy assumption fails.

We point out that, since the existence of at least one $\bar{M}$ for which operator $T_4[\bar{M}]$ is disconjugate on $[0,1]$ implies the validity of  Theorem \ref{T::5},
 operator $T_4[M]$ cannot be disconjugate on $[0,1]$ for any real parameter $M$ and not only for $\bar{M}=0$. 	
	\vspace{0.5cm}
		
	First, we obtain the Green's function expression related to the operator $T_4[0]\,u(t)$ in $X_3$, $g_0(t,s)$. By means of the Mathematica package developed in \cite{CCM}, we have that if follows  the expression
{\tiny	\[
	\left\lbrace  
	\begin{array}{cc}
	\frac{e^{10 (t-s)}-\frac{e^{-5 (2 s+t)} \left(-3 e^{10 s+5}+2 e^{15 s} \cos \left(5 \sqrt{3} (s-1)\right)+e^{15}\right) \left(-3 e^{5 t}+e^{15 t}+2 \cos \left(5 \sqrt{3}
			t\right)\right)}{-3 e^5+e^{15}+2 \cos \left(5 \sqrt{3}\right)}+2 e^{5 s-5 t} \cos \left(5 \sqrt{3} (t-s)\right)-3}{3000}, & 0\leq s\leq t\leq 1, \\\\\\
	-\frac{e^{-5 (2 s+t)} \left(-3 e^{10 s+5}+2 e^{15 s} \cos \left(5 \sqrt{3} (s-1)\right)+e^{15}\right) \left(-3 e^{5 t}+e^{15 t}+2 \cos \left(5 \sqrt{3} t\right)\right)}{3000
		\left(-3 e^5+e^{15}+2 \cos \left(5 \sqrt{3}\right)\right)}, & 0<t<s\leq 1.
	\end{array}
\right. \]}

Let us see now that $g_0(t,s)\leq 0$ on $[0,1]\times [0,1]$ and that it satisfies condition $(N_g)$, i.e., the following inequality is satisfied
\[\dfrac{g_0(t,s)}{t^3\,(t-1)}>0\,,\quad \mbox{for all   }  (t,s)\in[0,1]\times (0,1)\,.\]

To study the behaviour on a neighborhood of $t=0$ and $t=1$, we define the following functions

{\tiny \begin{eqnarray}\nonumber
k_1(s)=\lim_{t\rightarrow 0^+}\dfrac{g_0(t,s)}{t^3\,(t-1)}&=&\dfrac{e^{-10 s} \left(-3 e^{10 s+5}+2 e^{15 s} \cos \left(5 \sqrt{3} (s-1)\right)+e^{15}\right)}{6 \left(-3 e^5+e^{15}+2 \cos \left(5 \sqrt{3}\right)\right)},\\\nonumber\\\nonumber
k_2(s)=\lim_{t\rightarrow 1^-}\dfrac{g_0(t,s)}{t^3\,(t-1)}&=&  \frac{1}{300} e^{-10 s-5} \left(e^{15 s} \left(\sqrt{3} \sin \left(5 \sqrt{3} (s-1)\right)-\cos \left(5 \sqrt{3}
(s-1)\right)\right)+e^{15} {\color{white}\dfrac{e^{15}}{e^{15}} }\right. \\\nonumber\\\nonumber&&\left.+ \frac{\left(-3 e^{10 s+5}+2 e^{15 s} \cos \left(5 \sqrt{3} (s-1)\right)+e^{15}\right) \left(-e^{15}+\sqrt{3} \sin \left(5 \sqrt{3}\right)+\cos
	\left(5 \sqrt{3}\right)\right)}{-3 e^5+e^{15}+2 \cos \left(5 \sqrt{3}\right)}\right)\,.\end{eqnarray}}

In the sequel we will prove that both functions are strictly positive on $(0,1)$. 

It is not difficult to verify that $k_1(1)=k_1'(1)=k_1''(1)=0$ and that \[k_1^{(3)}(1)=-\frac{500 e^5}{-3 e^5+e^{15}+2 \cos \left(5 \sqrt{3}\right)}<0\,.\]

 If we prove that $k_1^{(3)}(s)$ is strictly negative on $[0,1]$, since, in such a case, $k_1''(s)$ would be positive and $k_1'(s)$ negative, we will deduce that $k_1(s)>0$ for $s\in(0,1)$.  
 
 Due to the fact that
\[k_1^{(3)}(s)=-\frac{500 e^{-10 s} \left(2 e^{15 s} \cos \left(5 \sqrt{3} (s-1)\right)+e^{15}\right)}{3 \left(-3 e^5+e^{15}+2 \cos \left(5 \sqrt{3}\right)\right)}\,,\]
we only must check that
\[{k_1}_1(s):=2 e^{15 s} \cos \left(5 \sqrt{3} (s-1)\right)+e^{15}>0\,,\quad s\in[0,1]\,.\]

But previous inequality holds immediately from the fact that 
\[\min_{s \in [0,1]}{k_1}_1(s)=e^{15}\left(1-e^{-\frac{2 \pi}{ \sqrt{3}}}\right) >0\,,\quad s\in[0,1]\,.\]

Consider now function $k_2$. We have that $k_2(0)=0$ and 
\[k_2'(0)=\frac{1+e^5 \left(e^{15}-\sqrt{3} \left(e^{10}-1\right) \sin \left(5 \sqrt{3}\right)-\left(1+e^{10}\right) \cos \left(5 \sqrt{3}\right)\right)}{10 e^5 \left(-3 e^5+e^{15}+2 \cos
	\left(5 \sqrt{3}\right)\right)}>0\,.\] 

So, we study the sign of its first derivative
{ \begin{eqnarray}\nonumber k_2'(s)&=&\frac{e^{-10 s-5}}{30 \left(-3 e^5+e^{15}+2 \cos \left(5 \sqrt{3}\right)\right)}\,{k_2}_0(s), \end{eqnarray}}
with
{ 
\begin{eqnarray}\nonumber
{k_2}_0(s)&=&e^{15 s} \left(\sqrt{3} \left(e^5 \left(2 e^{10}-3\right) \sin \left(5 \sqrt{3} (s-1)\right)+\sin \left(5 \sqrt{3} s\right)\right)
-3 e^5 \cos \left(5 \sqrt{3} (s-1)\right) \right.\\
&& \nonumber \left.+3
\cos \left(5 \sqrt{3} s\right)\right) +e^{15} \left(3 e^5-\sqrt{3} \sin \left(5 \sqrt{3}\right)-3 \cos \left(5 \sqrt{3}\right)\right).\end{eqnarray}}

It is clear that such function satisfies
\[{k_2}_0(s)>\left(-3-3 e^5+\sqrt{3} \left(-1+3 e^5-2 e^{15}\right)\right) e^{15 s}+e^{15} \left(3 e^5-\sqrt{3} \sin \left(5 \sqrt{3}\right)-3 \cos \left(5 \sqrt{3}\right)\right),\]
which is positive for 

{\scriptsize $s<\frac{1}{15} \left(\log \left(3 e^{20}-\sqrt{3} e^{15} \sin \left(5 \sqrt{3}\right)-3 e^{15} \cos \left(5 \sqrt{3}\right)\right)-\log \left(3+\sqrt{3}+3 e^5-3 \sqrt{3} e^5+2
\sqrt{3} e^{15}\right)\right)\approxeq0.32389$}.

Moreover, for $s\in[1-\frac{2 \pi }{5 \sqrt{3}},1-\frac{ \pi }{5 \sqrt{3}}]\approxeq[0.27448,0.63724]$ we have that
\[{k_2}_0(s)>\left(-4-3 e^5\right) e^{15 s}+e^{15} \left(3 e^5-\sqrt{3} \sin \left(5 \sqrt{3}\right)-3 \cos \left(5 \sqrt{3}\right)\right)\,,\]
and right part of previous equality is positive for \[s<\frac{1}{15} \left(\log \left(3 e^{20}-\sqrt{3} e^{15} \sin \left(5 \sqrt{3}\right)-3 e^{15} \cos \left(5 \sqrt{3}\right)\right)-\log \left(4+3 e^5\right)\right)\approxeq 0.99954.\]

Then, we have that $k_2'(s)>0$ for $s\in[0,1-\frac{ \pi }{5 \sqrt{3}}]$, and, as consequence, the same holds for $k_2(s)$.

On the other hand, we have that $k_2(1)=k_2'(1)=0$ and $k_2''(1)=1$, moreover
\[k_2''(s)=\frac{e^{-10 s-5}}{3 \left(-3 e^5+e^{15}+2 \cos \left(5 \sqrt{3}\right)\right)} {k_2}_1(s)\,,\]
where
{\tiny \begin{eqnarray}\nonumber {k_2}_1(s)&=&e^{15 s} \left(\sqrt{3} \left(e^{15} \sin \left(5 \sqrt{3} (s-1)\right)-\sin \left(5 \sqrt{3} s\right)\right)+3 e^5 \left(e^{10}-2\right) \cos \left(5 \sqrt{3} (s-1)\right)+3
\cos \left(5 \sqrt{3} s\right)\right)\\\nonumber&&+e^{15} \left(-3 e^5+\sqrt{3} \sin \left(5 \sqrt{3}\right)+3 \cos \left(5 \sqrt{3}\right)\right)\,.\end{eqnarray}}

Now, we must verify that ${k_2}_1(s)>0$. 

If $s>0.9$ we can bound it from below by the following function
{\tiny \[e^{15 s} \left(-3-\sqrt{3} \left(1+e^{15}\right)+3 e^5 \left(e^{10}-2\right) \cos \left(\frac{\sqrt{3}}{2}\right)\right)+e^{15} \left(-3 e^5+\sqrt{3} \sin \left(5
	\sqrt{3}\right)+3 \cos \left(5 \sqrt{3}\right)\right).\]}

It is clear that it is positive for $s\in(s_1,1]$, where
{\tiny \[s_1=\frac{1}{15} \log \left(\frac{-3 e^{20}+\sqrt{3} e^{15} \sin \left(5 \sqrt{3}\right)+3 e^{15} \cos \left(5 \sqrt{3}\right)}{3+\sqrt{3}+\sqrt{3} e^{15}+6 e^5 \cos
		\left(\frac{\sqrt{3}}{2}\right)-3 e^{15} \cos \left(\frac{\sqrt{3}}{2}\right)}\right)\approxeq 0.510335,\]}
which ensures that $k_2(s)>0$ on $(0.9,1)$.

On the other hand, for every $s\in [0,1]$, function
$300 \,e^{10 s+5}\,k_2(s)$ is bounded from below by
 \[{k_2}_2=\frac{1}{100} \left(-476 e^{15 s}+303 e^{10 s+5}-e^{15}\right)\,,\]
 which is positive for $s\in(s_2,s_3)$, where
{\scriptsize \[s_2=1+\frac{1}{5} \log \left(\frac{101}{476}+\frac{101}{476} \sqrt{3} \sin \left(\frac{1}{3} \tan ^{-1}\left(\frac{476 \sqrt{973657}}{917013}\right)\right)-\frac{101}{476} \cos
	\left(\frac{1}{3} \tan ^{-1}\left(\frac{476 \sqrt{973657}}{917013}\right)\right)\right)\approxeq 0.438593\,,\]}
and
{ \[s_3=1+\frac{1}{5} \log \left(\frac{101}{476}+\frac{101}{238} \cos \left(\frac{1}{3} \tan ^{-1}\left(\frac{476 \sqrt{973657}}{917013}\right)\right)\right)\approxeq0.908\,.\]}

So, we  conclude that $k_2(s)>0$ for every $s\in(0,1)$.\\[.3cm]

Now, in order to deduce condition $(N_g)$, we only have to verify that $g_0(t,s)<0$ for every $(t,s)\in(0,1)\times(0,1)$.

If $t<s$ we can express \[g_0(t,s)=-\frac{e^{-5 (2 s+t)} \,\ell_1(s)\,\ell_2(t)}{3000
	\left(-3 e^5+e^{15}+2 \cos \left(5 \sqrt{3}\right)\right)}\,,\]
where
\begin{eqnarray*}
\nonumber \ell_1(s)&=&\left(-3 e^{10 s+5}+2 e^{15 s} \cos \left(5 \sqrt{3} (s-1)\right)+e^{15}\right), \\\nonumber\\\nonumber
\ell_2(t)&=&\left(-3 e^{5 t}+e^{15 t}+2 \cos \left(5 \sqrt{3} t\right)\right).
\end{eqnarray*}

So, we must prove that both functions are positive on $(0,1)$.

 $\ell_1(s)$ is a positive multiple of $k_1(s)$, so, as we have proved before, it is positive for $s\in(0,1)$. 

To study the sign of $\ell_2$, since it satisfies that $\ell_2(0)=\ell_2'(0)=\ell_2''(0)=0$, from the following expressions, valid for all $t\in[0,1]$,

\[\ell_2^{(3)}(t)=375 \left(-e^{5 t}+9 e^{15 t}+2 \sqrt{3} \sin \left(5 \sqrt{3} t\right)\right)\geq 375 \left(-e^{5 t}+9 e^{15 t}-2 \sqrt{3} \right)>0,\]
we  deduce that $\ell_2(t)>0$ for every $t\in(0,1)$. \\

Let us see now what happens for $0<s\leq t<1$.

We can express $g_0(t,s)$ as follows
\[g_0(t,s)=\dfrac{1}{3000}\left( p_2(t-s)-p_1(t,s)\right) \,,\quad 0<s\leq t<1\,,\]
where 
\begin{eqnarray}
\nonumber p_1(t,s)&=&\dfrac{e^{-5(2s+t)}\,\ell_1(s)\,\ell_2(t)}{-3\,e^5+e^{15}+2\,\cos\left( 5\sqrt{3}\right) }
\end{eqnarray}
and
\begin{eqnarray*}
p_2(r)&=&e^{10\,r}+2\,e^{-5\,r}\,\cos\left( 5\,\sqrt{3}\,r\right) -3\,.
\end{eqnarray*}

From the previously proved positiveness of $\ell_1$ and $\ell_2$, we know that  $p_1(t,s)>0$.

On the other hand, since  $p_2(0)=p_2'(0)=p_2''(0)=0$, if we verify that  $p_2^{(3)}(r)>0$ for every $r\in[0,1]$, then we  conclude that the same holds for $p_2$ on $(0,1]$. In this case

\[p_2^{(3)}(r)=1000 e^{10 r}+2000 e^{-5 r} \cos \left(5 \sqrt{3} r\right).\]

This function is trivially positive whenever $0 \le r\leq \frac{\pi }{10 \sqrt{3}}\approxeq0.18138$.

Moreover, for every $r\in[0,1]$, we have that
\[p_2^{(3)}(r)>1000 e^{10 r}-2000 e^{-5 r}\,,\]
which is positive if, and only if, $r>\frac{\log (2)}{15}\approxeq 0.0462$. 

As consequence we deduce that $p_2(r)>0$ for every $r\in(0,1]$.\\[.3cm]

Then if we prove that $p_2(t-s)<p_1(t,s)$ for $0<s\leq t<1$, we can conclude that $g_0(t,s)<0$.

Notice that, if we have two  strictly convex functions on a suitable interval, we may affirm that they  have at most two common points. In the sequel, to prove our result,  we use this property.

Since by definition $g_0(1,s)=0$, we know that $p_1(1,s)=p_2(1-s)$, for every fixed $s\in(0,1)$.

From the fact, proved before, that $k_2>0$ on $(0,1)$, we know that $g_0(t,s)<0$ on a neighborhood of $t=1$ for every $s\in(0,1)$. Then $p_1(t,s)>p_2(t-s)$ on a neighborhood of $t=1$ for every $s\in(0,1)$.

Let us see now that, for every $s\in(0,1)$, $p_1(t,s)$ and $p_2(t-s)$ are convex functions of $t$

By direct calculation, we have that
{ \[\dfrac{\partial^2}{\partial t^2}p_1(t,s)=\dfrac{100 e^{-5 (2 s+t)}  \left(e^{15 t}+\sqrt{3} \sin \left(5 \sqrt{3} t\right)-\cos \left(5 \sqrt{3}
t\right)\right)\, \ell_2(s)}{-3\,e^5+e^{15}+2\,\cos\left( 5\sqrt{3}\right) },\]}
so we only  need to verify that
\[{p_1}_1(t)=\left(e^{15 t}+\sqrt{3} \sin \left(5 \sqrt{3} t\right)-\cos \left(5 \sqrt{3}
t\right)\right)>0\,,\quad t\in (0,1)\,.\]

The following inequality is trivially fulfilled

\[{p_1}_1(t)>e^{15 t}+\sqrt{3} \sin \left(5 \sqrt{3} t\right)-1=q_1(t)\,,\quad t\in[0,1]\,.\]

We have that 
\[q_1'(t)=15\,e^{15\,t}+15\,\cos(5\,\sqrt{3}\,t)>15(e^{15\,t}-1)>0\,,\]
since $q_1(0)=0$, we conclude that $q_1>0$ and, as consequence, ${p_1}_1(t)>0$ on $(0,1]$ and also $\dfrac{\partial^2}{\partial t^2}p_1(t,s)>0$.

\vspace{0.8cm}

We have already proved that $p_2^{(3)}(r)>0$, for $r\in[0,1]$, and $p_2''(0)=0$, so for every fixed $s\in(0,1)$ $p_2''(t-s)>0$ for every $t\in(s,1]$.

As consequence, for any fixed $s\in(0,1)$, both $p_1(t,s)$ and $p_2(t-s)$ are convex functions of $t$.   

From the fact that $p_1(t,s)>p_2(t-s)$ on a neighborhood of $t=1$, $p_1(1,s)=p_2(1-s)$ and, also, $p_1(s,s)>0=p_2(0)$, we can affirm that $p_1(t,s)>p_2(t-s)$ for  $t\in[s,1)$, and then $g_0(t,s)<0$ for  $0<s\leq t<1$, and condition $(N_g)$ is fulfilled.\\[.3cm]

Now, as a consequence of Theorem \ref{T::7}, we know that $g_M(t,s)\leq 0$ for $M\in [0,-\lambda_1)$, where $\lambda_1<0$ is the biggest negative eigenvalue of $T_4[0]\,u(t)$ in $X_3$.

To verify that Theorem \ref{T::5} does not hold in this case we will prove that for $M<0$ the sign change does not come on the least positive eigenvalue of $T_4[0]\,u(t)$ in $X_2$.
	
	As in the previous section, we can obtain numerically the first eigenvalues of $T_4[0]$, which can be given by the following approximation values:
	\begin{itemize}
		\item The biggest negative eigenvalue in $X_1$ is $\lambda_3\approxeq-(12.529)^4$.
		\item The least positive eigenvalue in $X_2$ is $\lambda_2\approxeq(10.895)^4$.
		\item The biggest negative eigenvalue in $X_3$ in $\lambda_1\approxeq -(9.458)^4$.
	\end{itemize}

	\begin{remark}
	\label{r-non-dis}
	Realize that, since $T_4[0]\,u(t)=0$ is not disconjugate on $[0,1]$, we have no a priori information about the sign of the eigenvalues $\lambda_3$ and $\lambda_2$. However, since $g_0$ satisfies $(N_g)$, we can ensure, without calculate it, that  $\lambda_1<0$.
	\end{remark}
	
Finally, let's see that there exists $M^*>-\lambda_2$ for which $g_{M^*}$ has not constant sign on $I \times I$. 
	
We are going to study the following function
	\[v(t)=\dfrac{\partial }{\partial s}g_{M^*}(t,s)\mid_{s=0}\,.\]
As we have proved in the proof of Theorem \ref{T::5}, if this function has not constant sign on $I$ then the Green's function must necessarily change sing in a neighborhood of $s=0$.

For $M^*=-\frac{59584}{9} \approxeq -(9.02032)^4$, we have that $v(t)$ follows the next expression
{\tiny \begin{eqnarray}\nonumber 
	\frac{3 e^{-\frac{1}{3} \left(9+\sqrt{669}\right) t} \left(277 e^{6 t} \left(\left(213 \sqrt{669}-27875\right) e^{2 \sqrt{\frac{223}{3}} t}-27875-213 \sqrt{669}\right)+446
		e^{\sqrt{\frac{223}{3}} t} \left(537 \sqrt{831} \sin \left(\sqrt{\frac{277}{3}} t\right)+34625 \cos \left(\sqrt{\frac{277}{3}} t\right)\right)\right)}{8441871944}&&\\\nonumber-\frac{223 \left(537 \sqrt{831} \sin \left(\sqrt{\frac{277}{3}}\right)+34625 \cos \left(\sqrt{\frac{277}{3}}\right)\right)+277 e^6 \left(213 \sqrt{669} \sinh
		\left(\sqrt{\frac{223}{3}}\right)-27875 \cosh \left(\sqrt{\frac{223}{3}}\right)\right)}{8441871944 \left(-277 \left(2007+152 \sqrt{669}\right) e^6+277 \left(152
		\sqrt{669}-2007\right) e^{6+2 \sqrt{\frac{223}{3}}}+446 e^{\sqrt{\frac{223}{3}}} \left(2493 \cos \left(\sqrt{\frac{277}{3}}\right)-98 \sqrt{831} \sin
		\left(\sqrt{\frac{277}{3}}\right)\right)\right)}&&\\\nonumber
			6 e^{\sqrt{\frac{223}{3}}-\frac{1}{3} \left(9+\sqrt{669}\right) t} \left(277 e^{6 t} \left(\left(152 \sqrt{669}-2007\right) e^{2 \sqrt{\frac{223}{3}} t}-2007-152
			\sqrt{669}\right)+446 e^{\sqrt{\frac{223}{3}} t} \left(2493 \cos \left(\sqrt{\frac{277}{3}} t\right)-98 \sqrt{831} \sin \left(\sqrt{\frac{277}{3}} t\right)\right)\right)\,,&&\end{eqnarray}}
which, see Figure \ref{v-changes}, changes sign on $I$.

\begin{figure}[h]
\centering
\includegraphics[width=0.3\textwidth]{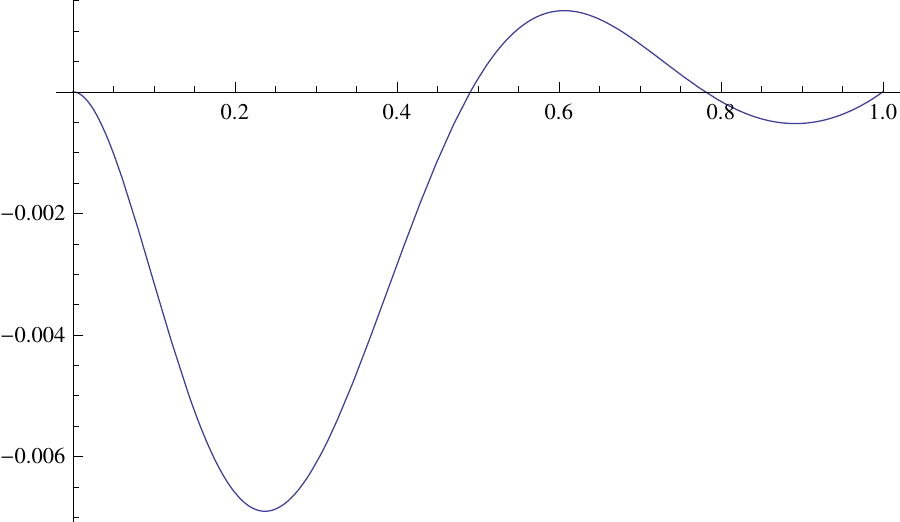}
\caption{\scriptsize{Graph of $v$}\label{v-changes}}
\end{figure}

As consequence the Green's function has not constant sign for a value of $M$ bigger than $-\lambda_2$.

	Even more, we can verify numerically which is the interval for $M$, where $g_M(t,s)$ is non-positive on $I \times I$. We  observe that change sign come first on the interior of $I\times I$.
	It comes in $(t,s)\approxeq (0.7186,0.0307)\in (0,1)\times (0,1)$ for $M \approxeq -(7.87022)^4$. So we deduce that it is given by $[-(7.87022)^4,-\lambda_1)$.
	
	As consequence we conclude the example that show us that  if we suppress the disconjugacy hypothesis, Theorem \ref{T::5} is not true in general.


\begin{thebibliography}{99}\addcontentsline{toc}{theorem}{Bibliography}
	 \bibitem{Cab} A. Cabada, \emph{Green's Functions in the Theory of Ordinary Differential Equations}, Springer Briefs in Mathematics, 2014.
	 
	 \bibitem{A2}
A. Cabada, \emph{The method of lower and upper solutions for second, third, fourth, and higher order boundary value problems}, J. Math. Anal. Appl. 185 (1994) 302-320.

 \bibitem{caci} A. Cabada, J. A. Cid, \emph{Existence and multiplicity of solutions for a periodic Hill's equation with parametric dependence and singularities}, Abstr. Appl. Anal. 2011, Art. ID 545264, 19 pp.

\bibitem{caci2} A. Cabada, J. A. Cid, \emph{Existence of a non-zero fixed point for non-decreasing operators
via Krasnoselskii's fixed point theorem,} {Nonlinear Anal.}, {71} (2009), 2114--2118.

\bibitem{CCI} A. Cabada, J. A. Cid, G. Infante, \emph{New criteria for the existence of non-trivial fixed points in cones}, {Fixed Point Theory and Appl.}, {2013:125}, (2013), 12 pp.
 


	 \bibitem{CCM}
	 A. Cabada, J.A. Cid, B. M\'aquez-Villamar\' {\i}n, \emph{Computation of Green's functions for boundary value problems with Mathematica},
	 Applied Mathematics and Computation 219 (2012) 1919-1936.
	 
	 \bibitem{cacisa}
A. Cabada, J.A. Cid, L. Sanchez, \emph{Positivity and lower and upper solutions for fourth order boundary value problems}, Nonlinear Anal. 67 (2007), 1599-1612.


	 \bibitem{CE}
	 A. Cabada, R. R. Engui\c{c}a \emph{Positive solutions of fourth order problems with clamped beam boundary conditions}, Nonlinear Anal. 74 (2011), 3112-3122.
	 
	 \bibitem{CaFe}
	 A. Cabada, C. Fern\'andez-G\'omez \emph{Constant Sign Solutions of Two-Point Fourth Order Problems}, Appl. Math. Comput. 263 (2015), 122-133. 
	 
	 \bibitem{CaSaa}
	 A. Cabada, L. Saavedra \emph{Disconjugacy characterization by means of spectral $(k,n-k)$ problems}, Appl. Math. Lett. 52 (2016), 21-29.
	 
	 \bibitem{Cop} W. A. Coppel, \emph{Disconjugacy}, Lecture Notes in Mathematics, Vol. 220. Springer-Verlag, Berlin-New York, 1971.
	 
	 	 \bibitem{jc-df-fm} J. A. Cid, D. Franco, F. Minh\'{o}s, \emph{Positive fixed points and fourth-order equations}, {Bull. Lond. Math. Soc.},
{41} (2009), 72--78.

	 \bibitem{CoHa} {C. De Coster, P. Habets,}  \emph{Two-Point Boundary Value Problems: Lower and Upper Solutions}, {Mathematics in Science and Engineering {205}, Elsevier B. V., Amsterdam, 2006.}
	 
	 
	 \bibitem{Eli} {U. Elias,} \emph{Eventual disconjugacy of $y^{(n)}+\mu\,p(x)\,y=0$ for every $\mu$}, Arch. Math. (Brno) 40 (2004), 2, 193--200. 
	 
	 \bibitem{Erbe} {L. Erbe,} \emph{Hille-Wintner type comparison theorem for selfadjoint fourth order linear differential equations}, Proc. Amer. Math. Soc. 80 (1980), 3, 417--422. 
	 
	 \bibitem{df-gi-jp-prse}
D. Franco, G. Infante, J. Per\'an, \emph{A new criterion for the existence of multiple solutions in cones}, {Proc. Roy. Soc. Edinburgh Sect. A}, {142} (2012), 1043--1050.

	 \bibitem{graefkongwangAML} J. R. Graef, L. Kong, H. Wang,
\emph{A periodic boundary value problem with vanishing Green's function},
{Applied Mathematics Letters} {21} (2008), 176-180.

\bibitem{graefkongwang} J. R. Graef, L. Kong, H. Wang, \emph{Existence, multiplicity, and dependence on a
parameter for a periodic boundary value problem}, {J. Differential Equations} {245} (2008), 1185-1197.

\bibitem{krasnoselskii} M. A. Krasnosel'ski\u{\i},  \emph{Positive Solutions of Operator Equations},
Noordhoff, Groningen, 1964.

	 \bibitem{KwZe} {M. K. Kwong, A. Zettl,} \emph{Asymptotically constant functions and second order linear oscillation}, J. Math. Anal. Appl. 93 (1983), 2, 475--494. 
	 
	 \bibitem{llv} {G. S. Ladde, V. Lakshmikantham, A. S. Vatsala,} \emph{Monotone Iterative Techniques for Nonlinear Differential Equations}, { Pitman, Boston M. A. 1985.} 
	 
 \bibitem{lfb} 	 H. Li, Y. Feng, C. Bu, \emph{Non-conjugate boundary value problem of a third order differential equation},  Electron. J. Qual. Theory Differ. Equ. 2015, 21, 1-19
	 
	 \bibitem{ML} R. Ma, Y. Lu \emph{Disconjugacy and monotone iteration method for third-order equations}, Commun. Pure Appl. Anal., 13,  3, (2014), 1223-1236.
	 
	 \bibitem{persson} H. Persson, \emph{A fixed point theorem for monotone functions}, {Appl. Math. Lett.}, {19} (2006) 1207--1209.
	 
	 \bibitem{Si} {W. Simons,}  \emph{Some disconjugacy criteria for selfadjoint linear differential equations}, J. Math. Anal. Appl. 34 (1971) 445--463.
	 
	 \bibitem{to} { P. J. Torres,}  \emph{Existence of One-Signed Periodic Solutions of Some Second-Order Differential Equations via a Krasnoselskii Fixed Point Theorem}, J. Differential Equations {190} (2003), 2, 643 -- 662.
	 
	 
	 \bibitem{Ze1} {A. Zettl,}  \emph{A constructive characterization of disconjugacy}, Bull. Amer. Math. Soc. 81 (1975), 145--147. 
	 
	 \bibitem{Ze2} {A. Zettl,}  \emph{A characterization of the factors of ordinary linear differential operators}, Bull. Amer. Math. Soc. 80 (1974), 498--499.
	 




	 
	\end{thebibliography}
\end{document}